\documentclass[11pt]{amsart}
\usepackage{amsmath,amssymb,amscd,amsthm}
\usepackage{amsthm}
\usepackage[margin=1.3in]{geometry}
\usepackage[pagewise]{lineno}

\usepackage{tikz-cd}
\usetikzlibrary{babel}
\usetikzlibrary{decorations.pathmorphing}
\usepackage[margin=1.3in]{geometry}
\usepackage{hyperref}
\hypersetup{hypertex=true,
            colorlinks=true,
            linkcolor=blue,
            anchorcolor=blue,
            citecolor=blue}

\usepackage[all]{xy}

\usepackage{enumerate}
\usepackage{paralist}
\usepackage{amscd}

\newtheorem{thm}{Theorem}[section]
\newtheorem{lem}[thm]{Lemma}
\newtheorem{prop}[thm]{Proposition}

\newtheorem{cor}[thm]{Corollary}

\newtheorem*{prob}{\bf Problem}
\theoremstyle{definition}\newtheorem{df}[thm]{Definition}
\theoremstyle{definition}\newtheorem{rem}[thm]{Remark}
\theoremstyle{definition}

\theoremstyle{definition}

\renewcommand{\phi}{\varphi}

\newcommand{\Homeo}{\mathcal{H}}

\newcommand{\Z}{\mathbb{Z}}

\newcommand{\Sph}{\mathbb{S}}

\newcommand{\Raf}{\mathcal{R}_0(\tilde{\alpha})}
\newcommand{\Rbt}{\mathcal{R}_0(\tilde{\beta})}
\newcommand{\C}{\mathbb{C}}
\newcommand{\T}{\mathbb{T}}
\newcommand{\taf}{\tilde{\alpha}}
\newcommand{\tbt}{\tilde{\beta}}

\newcommand{\Aff}{\operatorname{Aff}}
\newcommand{\Pj}{\mathcal{P}}

\newcommand{\id}{\operatorname{id}}

\newcommand{\ep}{\varepsilon}
\newcommand{\F}{{\mathcal F}}

\newcommand{\Lip}{\mathcal{L}}
\newcommand{\G}{\mathcal{G}}

\newcommand{\af}{{\alpha}}
\newcommand{\bt}{{\beta}}

\newcommand{\beq}{\begin{eqnarray}}
\newcommand{\eneq}{\end{eqnarray}}

\begin{document}
\title{Approximate $K$-conjugacies and $C^*$-approximate conjugacies of minimal dynamical systems}

\author{Sihan Wei}

\maketitle
\begin{center}
Research Center for Operator Algebras, 

School of Mathematics and Science, 

East China Normal University, 

Shanghai, China
\end{center}

\begin{abstract} 
In this article, we extend H. Matui and H. Lin's notions of approximate $K$-conjugacies and $C^*$-strongly approximate conjugacies to general minimal dynamical systems. In particular, upon modifying a result of the existence of minimal skew products, we answer a question of H. Lin and show that, associated with any Cantor minimal system $(K,\af)$, there is a class $\Raf$ of minimal skew products on $K\times \Omega$, such that for any two rigid homeomorphisms $\af\in\Raf$ and $\bt\in\Rbt$, the notions of approximate $K$-conjugacy and $C^*$-strongly approximate conjugacy coincide, which are also equivalent to a $K$-version of Tomiyama's commutative diagram, where $\Omega$ is an (infinite) connected finite CW-complex with torsion free $K$-groups and the so-called {\em Lipschitz-minimal-property (LMP)}.

\end{abstract}
\section{Introduction}
The interplay between the theory of $C^*$-algebras and topological dynamical systems, perhaps, began with J. Tomiyama's works on the relations of orbital properties of dynamical systems and ideal properties of the associated $C^*$-algebras (see \cite{JT}, for example), and the relations of conjugacies of topologically free dynamical systems and isomorphisms of $C^*$-algebras in \cite{JT2}. J. Tomiyama shows in \cite{JT2} that, for two topologically free dynamical systems $(X,\sigma)$ and $(Y,\tau)$, their crossed products $C^*(X,\sigma)$ and $C^*(Y,\tau)$ are isomorphic keeping their commutative subalgebras $C(X)$ and $C(Y)$ if and only if $(X,\sigma)$ and $C^*(Y,\tau)$ are flip conjugate. This celebrated theorem can be summarized by an equivalence of following commutative diagrams,

\begin{equation}
\begin{tikzcd}\label{Tu}
X   \ar[r, "\sigma"]  \ar[d, rightarrow, "h"'] \ar[dr, phantom, "\circlearrowleft"] & X \ar[d, rightarrow, "h"]\\
Y \ar[r, rightarrow, "\tau({\rm or}\ \tau^{-1})"] & Y
\end{tikzcd}
\hspace{1cm}\Longleftrightarrow\hspace{0.6cm}
\begin{tikzcd}
C(X)   \ar[r, "\varphi|_{C(X)}"]  \ar[d, hookrightarrow, "j_\sigma"'] \ar[dr, phantom, "\circlearrowleft"] & C(Y) \ar[d, hookrightarrow, "j_\tau"]\\
C^*(X,\sigma) \ar[r, rightarrow, "\varphi"] & C^*(Y,\tau)
\end{tikzcd}
                         \end{equation}                 
where $j_\cdot$ denotes the canonical embedding of the $C^*$-algebra of continuous functions into the associated crossed product. For minimal dynamical systems on the Cantor space (which we called, a Cantor minimal system), T. Giordano, I. F. Putnam and C. F. Skau showed in \cite{GPS} that, the strong orbit equivalence, isomorphisms of crossed products and isomorphisms of their ordered $K_0$-group with distinguished order unit coincide. For more details of minimal $\Z$-actions on the Cantor space, see \cite{GHH}, \cite{HPS}, \cite{IP}, \cite{IP2}, \cite{LP3}, and \cite{GPS2}, \cite{GMPS}, \cite{EM} or \cite{P2} of free minimal $\Z^d$-actions on the Cantor space.

On the other hand, inheriting the philosophy of Tomiyama's classification theorem, H. Lin and H. Matui define the concepts of approximate $K$-conjugacy and $C^*$-strongly approximate conjugacy in \cite{LM1}, and they show in \cite{LM1} that for  Cantor minimal systems, the approximate $K$-conjugacy and the $C^*$-strongly approximate conjugacy coincide, which are also equivalent to the strong orbit equivalence defined by T. Giordano, I. F. Putnam and C. F. Skau in \cite{GPS}. In \cite{LM} and \cite{LM2}, minimal dynamical systems on the product of the Cantor space and the circle are considered, and it is shown that, for rigid homeomorphisms (the condition that the crossed product has tracial rank zero), not only the approximate $K$-conjugacy is equivalent to the $C^*$-strongly approximate conjugacy, they are also equivalent to a $K$-version of Tomiyama's commutative diagram. This is also the case for the minimal rigid dynamical systems on the product of the Cantor space and the torus, as is shown by W. Sun in \cite{WS}. Therefore, in \cite{LP}, H. Lin asked the following question that, for two minimal rigid dynamical systems $(X_1, h_1)$ and $(X_2,h_2)$, what additional hypothesis are required for $h_1$ and $h_2$ (and of course, for $X_1$ and $X_2$) such that the approximate $K$-conjugacy and the $C^*$-strongly approximate conjugacy are equivalent?

In this paper, we define the {\em Lipschitz-minimal-property (LMP)} for compact metric spaces (Definition \ref{4.2}). Then, upon modifying the technique developed by S. Glasner and B. Weiss in \cite{GW} for proving the existence of skew minimal products, we answer the question of H. Lin by showing that, for any Cantor minimal system $(K,\taf)$ and any infinite finite-dimensional connected finite CW-complex $\Omega$ with the LMP, there is an uncountable class $\Raf$ of minimal skew products on $K\times\Omega$ such that, with the additional condition that $K_i(C(\Omega))\,(i=0,1)$ is torsion free, for any two minimal rigid homeomorphisms $\af=\taf\times c_1\in\Raf$ and $\bt=\tbt\times c_2\in \Rbt$, the approximate $K$-conjugacy and the $C^*$-strongly approximate conjugacy coincide, which are also equivalent to the $K$-version of Tomiyama's commutative diagram (Theorem \ref{6.1}). This includes the cases $\Omega=\mathbb{T}$ considered by H. Lin in and H. Matui in \cite{LM}, $\Omega=\mathbb{T}^2$ considered by W. Sun in \cite{WS}. In addition, this is also the case if $\Omega$ is an even-dimensional sphere $\Sph^{2n}$ or a product of even dimensional spheres of different dimensions. However, note that the even-dimensional spheres admit no minimal homeomorphism.

In the noncommutative settings, we propose the following problem for a better understanding of the influence of $K$-theroy on $C^*$-dynamical systems.
\begin{prob}
Let $A$ be a unital simple $C^*$-algebra and $\af,\bt: A\to A$ two automorphisms. For two $C^*$-dynamical systems $(A,\af)$ and $(A,\bt)$, what additional hypothesis are required for $A, \af$ and $\bt$, such that there is a dynamically-interpretative equivalence relation of $(A,\af)$ and $(A,\bt)$ corresponding to the following Tomiyama's $K$-commutative diagram:
\begin{center}
\begin{tikzcd}
K_i(A)   \ar[r, "(\af)_{*i}"]  \ar[d, rightarrow, "(j_\alpha)_{*i}"'] \ar[dr, phantom, "\circlearrowleft"] & K_i(A) \ar[d, rightarrow, "(j_\beta)_{*i}"]\\
K_i(A\times_\af\Z) \ar[r, rightarrow, "\kappa_i"] & K_i(A\times_\bt\Z).
\end{tikzcd}
\end{center}
\end{prob}

The paper is organized as follows.  In Sect.~\ref{sec:2}, we introduce the notations. In Sect.~\ref{sec:3}, the notion of the rigidity of a minimal dynamical system (Definition \ref{3.5}) and a sufficient condition for a minimal dynamical system on $K\times\Omega$ to be rigid (Theorem \ref{3.7}) are given. In Sect.~\ref{sec:4}, we define the {\em Lipschitz-minimal-property (LMP)} for compact metric spaces (Definition \ref{4.2}) and construct an uncountable class $\Raf$ of minimal homeomorphisms on $K\times\Omega$ associated with any Cantor minimal system $(K,\taf)$ (Theorem \ref{4.5}). In Sect.~\ref{sec:5}, we concentrate on the concepts of approximate $K$-conjugacy and $C^*$-strongly approximate conjugacy, and prove a sufficient and necessary condition for two minimal homeomorphisms on $K\times\Omega$ to be approximately $K$-conjugate, which also unifies the various definitions of approximate $K$-conjugacies made by H. Lin (Theorem \ref{5.13}). In Sect.~\ref{sec:6} we give a proof of the main theorem (Theorem \ref{6.1}). In Sect.~\ref{sec:7}, we consider the the situation under which $K_1(C(\Omega))$ is not assumed to be trivial and prove another version of the main theorem (Theorem \ref{7.9}) with additional conditions on the orbit-breaking subalgebras.
\vspace{0.25cm}

{\bf Acknowledgements.} \ The author was partially supported by Shanghai Key Laboratory of PMMP, Science and Technology Commission of Shanghai Municipality (STCSM), grant $\#$13dz2260400 and a NNSF grant (11531003). The author is grateful to Professor H. Lin for discussions and his helpful advices. The main problems under consideration (Theorem \ref{6.1} and Theorem \ref{7.9})  was proposed by H. Lin and the author in 2019.

\section{Preliminaries}\label{sec:2}
\begin{df}
Let $A$ be a unital $C^*$-algebras. For a subset $\mathcal{F}\subset A$, we use $C^*(\mathcal{F})$ to denote the $C^*$-subalgebra of $A$ generated by $\mathcal{F}$. We use $U(A)$ and $\mathcal{P}(A)$ to denote the unitary group of $A$ and the set of projections in $A$. Denote $\mathcal{P}_\infty(A)=\bigcup_{n\ge1}\mathcal{P}(M_n(A))$.

For $\ep>0$, a subset $\F\subset A$ and an element $a\in A$, we write $a\in_{\ep}\F$, if $\|a-b\|<\ep$ for some $b\in\F$. We use $\mathcal{I}^{(k)}\,(k\geq1)$ to denote the class of all unital $C^*$-algebras which are unital hereditary $C^*$-subalgebras of the form $C(X)\otimes F$, where $X$ is a $k$-dimensional finite CW-complex and $F$ is a finite-dimensional $C^*$-algebra. In particular, we denote $\mathcal{I}^{(0)}$ the class of all finite-dimensional $C^*$-algebras.
\end{df}
\begin{df}[Definition 3.6.2, \cite{L1}]
A unital simple $C^*$-algebra $A$ is said to have tracial (topological) rank no more than $k$ if for any $\ep>0$, any finite set $\F\subset A$ and any nonzero element $a\in A_+$, there exist a nonzero projection $p\in A$ and a $C^*$-subalgebra $B\in\mathcal{I}^{(k)}$ with $1_B=p$ such that

(1) $\|px-xp\|<\ep$ for all $x\in\F$,

(2) $pxp\in_{\ep}B$ for all $x\in\F$ and

(3) $[1-p]\le[a]$, or equivalently, $1-p$ is equivalent to a projection in $Her(a)$.
\vspace{0.1cm}

We write $TR(A)\leq k$ if $A$ has tracial topological rank no more than $k$. If furthermore, $TR(A)\nleq k-1$, then we say $TR(A)=k$. It is also known that $TR(A)=0$ implies $RR(A)=0$ and ${\rm tsr}(A)=1$.
\end{df}
\begin{df}
For a unital $C^*$-algebra, we denote $T(A)$ the convex cone of tracial states of $A$, and ${\Aff}(T(A))$ the space of real continuous affine functions on $T(A)$. The $K_0(A)$ is equipped with the order unit $[1_A]$ and the positive cone $K_0(A)_+$. For $p\in\Pj(M_k(A))$ and $\tau\in T(A)$, we write $\tau_*([p])=\tau\otimes {\rm Tr}(p)$, 
where ${\rm Tr}(\cdot)$ is the standard (unnormalized) trace on $M_k(\C)$. The homomorphism $\rho_A: K_0(A)\to \Aff(T(A))$ is defined by
\[\rho_A(x)(\tau)=\tau_*([p])-\tau_*([q])\]
where $x=[p]-[q]\in K_0(A)$. By abuse of notation, for $a\in M_\infty(A)_{s.a.}$, we use $\rho_A(a)$ to denote the map which sends $\tau\in T(A)$ to $\tau\otimes {\rm Tr}(a)$.
\end{df}
\begin{df}
Let $G$ and $F$ be abelian groups. An extension of $F$ by $G$ is said to be pure, if every finitely generated subgroup of $F$ lifts. We use $Pext(G,F)$ to denote the group of pure extensions of $F$ by $G$ in $ext_\mathbb{Z} (G,F)$. For $C^*$-algebras $A$ and $B$, denote $KL(A,B)=KK(A,B)/Pext(K_{*-1}(A), K_{*-1}(B))$. Let $C_n$ be a commutative $C^*$-algebra with $K_0(C_n)=\Z/n\Z$ and $K_1(C_n)=0$. Denote $K_i(A,\Z/n\Z)=K_i(A\otimes C_n)$ and $\underline{K}(A)=\bigoplus_{i=0,1}\bigoplus_{n=1}^\infty K_i(A,\Z/n\Z)$. We remark that if $A$ satisfies the UCT, then 
\[KL(A,B)\cong {\rm Hom}_{\Lambda}(\underline{K}(A),\underline{K}(B)),\]
where ${\rm Hom}_{\Lambda}(\underline{K}(A),\underline{K}(B))$ is the group of homomorphisms from $\underline{K}(A)$ to $\underline{K}(B)$ associated with the direct sum decomposition and the so-called Bockstein operations (see Definition 5.7.13, Definition 5.7.14, Definition 5.8.13 in \cite{L1}).
\end{df}
\begin{df}
Let $A$ and $B$ be unital separable $C^*$-algebras and $\phi_n: A\to B$ be a sequence of contractive completely positive linear maps. We shall call $\{\phi_n\}$ a sequential asymptotic morphism, if
\[\lim_{n\to\infty}\|\phi_n(ab)-\phi_n(a)\phi_n(b)\|=0\]
for all $a,b\in A$. Note that for every projection $p\in M_n(A)$, $\phi_n(p)$ is close to a projection $q\in M_k(B)$ for sufficiently large $n$, and hence they are equivalent. We then use $[\phi_n(p)]=[q]$ to denote the equivalence.
\end{df}
\begin{df}
Let $X$ be an infinite compact metric space. By an (invertible) dynamical system on $X$, we mean the triple $(X,\af,\Z)$ where $\af$ is an action of $\Z$ on $X$ by an iteration of a single homeomorphism, which we still denote by $\af$. We then simply write $(X,\af)$ to represent a dynamical system on $X$. A dynamical system $(X,\af)$ is said to be minimal, if there is no nonempty proper $\af$-invariant closed subset of $X$, which is equivalent to the condition that every $\af$-orbit $Orb_\af(x)$ of $x\in X$ is dense in $X$. 

Throughout the paper, we use $C^*(X,\af)$ to denote the unital separable amenable $C^*$-algebra associated with a dynamical system $(X,\af)$. A dynamical system $(X,\af)$ is minimal if and only if $C^*(X,\af)$ is simple (see Theorem 4.3.3 in \cite{JT}).
\end{df}
\begin{df}
By a Cantor space, we mean a totally disconnected, perfect, compact metrizable space $K$. It is well-known that any two  Cantor spaces are homeomorphic, see for example, Theorem 2.12 in \cite{IP}. 

From now on, we will always use $K$ to denote the Cantor space. For a compact metric space $X$, we denote the group of homeomorphisms of $X$ onto itself by $\Homeo(X)$.  Endowing $\Homeo(X)$ with the metric $d(\cdot,\cdot)$ defined by
\[d(\af,\bt)=\sup_{x\in X}d_X(\af(x),\bt(x))+\sup_{x\in X}d_X(\af^{-1}(x),\bt^{-1}(x))\]
makes $\Homeo(X)$ a complete metric space.

From the totally disconnectedness of the Cantor space $K$, it is immediate that, for every compact connected metric space $\Omega$ and every homeomorphism $\af:K\times\Omega\to K\times\Omega$, there exists a homeomorphism $\tilde{\af}$ of $K$ onto itself, and a continuous map $c:K\to \Homeo(\Omega)$ such that 
\[\af(x,\omega)=(\tilde{\af}(x),c_x(\omega))\]
for all $(x,\omega)\in K\times\Omega$. We call the continuous map $c$ an $\tilde{\af}$-coboundary and write $\af=\tilde{\af}\times c$,
if $\tilde{\af}:K\to K$ and $c:K\to\Homeo(\Omega)$ is induced by $\af$.
\end{df}
\begin{df}[\cite{HPS}]
Let $(K,\tilde{\af})$ be a Cantor minimal system. A Kakutani-Rohlin partition of $(K,\af)$ is a partition of $K$ into nonempty disjoint clopen sets of $K$:
\[\Pj=\{X(v,j): v\in V, 1\leq j\leq h(v)\}\]
indexed by a finite set $V$ (of towers) and positive integers $h(v)$, called the height of the tower $v\in V$ satisfying 
\[\tilde{\af}(X(v,j))=X(v,j+1)\,(1\leq j\leq h(v)-1)\ \ {\rm and}\ \ \bigcup_{v\in V}X(v,1)=\tilde{\af}(\mathcal{R}(\Pj)),\]
where $\mathcal{R}(\Pj)=\bigcup_{v\in V}X(v,h(v))$ is called the roof set of $\Pj$.
\end{df}
\vspace{0.3cm}

\section{Orbit-breaking subalgebras and rigidity}\label{sec:3}
Throughout this section, we assume $K$ is the Cantor space, $\Omega$ admits the structure of a  finite-dimensional connected finite CW-complex and let $X=K\times\Omega$. For a  minimal homeomorphism $\af:X\to X$, write $A_\af=C^*(X,\af)$. Let $J\subset K$ be a closed set, we denote
\[A_J=C^*\{C(X), u_\af C_0((K\setminus J)\times\Omega)\},\]
where $u_\af$ is the implementing unitary of $\af$.
\begin{thm}\label{3.1}
Let $x\in K$. Then $A_{\{x\}}$ is a simple unital AH-algebra with no dimension growth. As a Consequence, we have $TR(A_{\{x\}})\leq1$.
\end{thm}
\begin{proof}
Let $\alpha=\tilde{\alpha}\times c$ where $\tilde{\alpha}\in\mathcal{H}(K)$ and $c\in C(K,\mathcal{H}(\Omega))$. Then $(K,\tilde{\alpha})$ is a Cantor minimal system. Let $\mathcal{P}_n=\{X(n,v,j)\subset K: v\in V_n, j=1,2,\cdots, h_n(v)\}$
be a sequence of Kakutani-Rohlin partitions associated with $(K,\tilde{\alpha})$, such that the roof sets $\mathcal{R}(\mathcal{P}_n)=\bigcup_{v\in V_n} X(n,v,h_n(v))$
shrinks to the singleton $\{x\}$ and $\bigcup\mathcal{P}_n$ generates the topology of $K$. Let
\[A_n=C^*\{C(K\times\Omega), u_{\alpha}C((K\setminus\mathcal{R}(\mathcal{P}_n))\times\Omega)\}\]
Note that $A_n\subset A_{n+1}\,(n\geq1)$, and every continuous function $f\in C_0((K\setminus\{x\})\times\Omega)$ can be approximated by elements in $C((K\setminus\mathcal{R}(\Pj_n))\times\Omega$ because $\{\mathcal{R}(\Pj_n)\}$ shrinks to $\{x\}$, which follows that $A_{\{x\}}$ is the norm closure of the union of all $A_n$'s, that is, $\displaystyle A_{\{x\}}=\lim_{\longrightarrow}A_n$. Additionally, since $K\setminus \mathcal{R}(\Pj_n)=\bigcup_{v\in V_n}\bigcup_{1\leq j\leq h_n(v)-1} X(n,v,j)$, 
\[C((K\setminus \mathcal{R}(\mathcal{P}_n))\times\Omega)\cong \bigoplus_{v\in V_n}\bigoplus_{1\leq j\leq h_n(v)-1} C(X(n,v,j))\otimes C(\Omega).\]
Now we show that, for every $n\geq1$, 
\[A_n\cong\bigoplus_{v\in V_n} M_{h_n(v)}\otimes C(X(n,v,h_n(v)))\otimes C(\Omega)\]
which is a direct sum of homogeneous $C^*$-algebras with topological dimension almost ${\rm dim}(K\times\Omega)={\rm dim}(\Omega)$.
This will follow that $A_{\{x\}}$ is an $AH$-algebra with no dimension growth, whence the theorem is proved. For $i,j\in \{1,\cdots, h_n(v)\}$, let 
\[e_{i,j}^{n,v}=u_{\alpha}^{i-j}1_{X(n,v,j)}.\]
It is easy to verify that $\{e_{i,j}^{n,v}\}_{i,j=1}^{h_n(v)}$ satisfy the relations of matrix units for all $v\in V_n$. In fact, for every $n\geq1$,
\begin{align*}
e_{i,j}^{n,v}\cdot e_{k,l}^{n,v}&=u_{\alpha}^{i-j}1_{X(n,v,j)}\cdot u_{\alpha}^{k-l}1_{X(n,v,l)}\\
&=u_\af^{i-j+k-l}\cdot 1_{\af^{l-k}(X(n,v,j))}\cdot 1_{X(n,v,l)}\\
&=u_\af^{i-j+k-l}\cdot 1_{X(n,v,j+l-k)}\cdot 1_{X(n,v,l)},
\end{align*}
and since $\{X(n,v,j): v\in V_n, j=1,\cdots, h_n(v)\}$ are disjoint, we have, if $j=k$, then $e_{i,j}^{n,v}\cdot e_{k,l}^{n,v}=u_\af^{i-l}\cdot 1_{X(n,v,l)}=e_{i,l}^{n,v}$, and if $j\neq k$, then $e_{i,j}^{n,v}\cdot e_{k,l}^{n,v}=0$, that is, 
\[e_{i,j}^{n,v}\cdot e_{k,l}^{n,v}=\delta_{j,k}\cdot e_{i,l}^{n,v}\]
which shows that $\{e_{i,j}^{n,v}\}_{i,j=1}^{h_n(v)}$ satisfy the relations of matrix units. Also note that if $v\ne w$, then $e_{i,j}^{n,v}\cdot e_{i',j'}^{n,w}=0$ for all $1\leq i,j\leq h_n(v), 1\leq i',j'\leq h_n(w)$. Now define projections
\[p_{n,v}=\sum_{j=1}^{h_n(v)}1_{X(n,v,j)}\]
for every $v\in V_n$. Since $p_{n,v}u_\af(1-1_{\mathcal{R}(\Pj_n)})=u_\af\sum_{j=1}^{h_n(v)-1}1_{X(n,v,j)}=u_\af(1-1_{\mathcal{R}(\Pj_n)})p_{n,v}$, $p_{n,v}$ is central in $A_n$ for every $n\geq1$. As we have shown, it is now clear that $\{e_{i,j}^{n,v}\}_{i,j=1,2,\cdots,h_n(v)}$ forms a matrix units for $p_{n,v}A_np_{n,v}$. On the other hand, according to the definition of $A_n$, it is straightforward to check that
\[1_{X(n,v,h_n(v)}A_n1_{X(n,v,h_n(v)}=C(X(n,v,h_n(v))\times\Omega)\]
which follows that $A_n\cong\bigoplus_{v\in V_n} M_{h_n(v)}\otimes C(X(n,v,h_n(v)))\otimes C(\Omega)$
as designed. As for the simplicity of $A_{\{x\}}$, it is based on the following observation. The unital $C^*$-algebra $A_\af$ corresponds to the groupoid $C^*$-algebra whose associated  equivalence relation is
\[\mathcal{R}=\{((x,\omega),\af^k(x,\omega)): (x,\omega)\in X, k\in\Z\}\]
while $A_{\{x\}}$ corresponds to
\[\mathcal{R}_x=\mathcal{R}\setminus\{\af^k(x,\omega),\af^l(x,\omega): (x,\omega)\in X, k\geq0,l\leq 0\ {\rm or}\ k\leq0,l\geq0\}.\]
It is well-known that any positive orbit of a minimal homeomorphism is dense in the underlying space, which follows that $\mathcal{R}_x$ is dense in $X$. Finally, according to Proposition 4.6 of \cite{R}, $A_{\{x\}}$ is simple.
This completes the proof.
\end{proof}
\begin{rem}\label{3.2}
In \cite{LM}, the case of $\Omega=\mathbb{T}$ is considered, and it is shown in \cite{LM} that in this case, $A_{\{x\}}$ is an A$\T$-algebra. In \cite{WS} and \cite{KS}, the case of $\Omega=\T^2$ and $\Omega=\Sph^{2n-1}$ are considered, and it is shown that $A_{\{x\}}$ is an AH-algebra with no dimension growth. Hence, one may also replace $\Sph^{2n-1}$ with $\Omega$ to get Theorem \ref{3.1}.
\end{rem}
For a Cantor minimal system $(K,\taf)$, we denote 
\[K^0(K,\taf)=C(K,\Z)/\{f-f\circ\taf^{-1}: f\in C(K,\Z)\}\]
which is unital positive isomorphic to $K_0(C^*(K,\taf))$.
\begin{prop}\label{3.3}
Let $\af=\tilde{\af}\times c:X\to X$ be a minimal homeomorphism. If $K_i(C(\Omega))$ is torsion free and $c_x\in\Homeo(\Omega)$ induces the identity map on $K_i(C(\Omega))$ for all $x\in X, i=0,1$, then $K_i(A_{\{x\}})$ and $K_i(A_\af)$ are torsion free and 
\[K_0(A_\af)\cong (K^0(K,\taf)\otimes K_0(C(\Omega)))\oplus K_1(C(\Omega))\]
and
\[K_1(A_\af)\cong (K^0(K,\taf)\otimes K_1(C(\Omega)))\oplus K_0(C(\Omega)).\]
Besides, we have $K_i(A_\af)\cong K_i(A_{\{x\}})\oplus K_{i+1}(C(\Omega))$ for $i=0,1$, and the canonical embedding $j_x: A_{\{x\}}\hookrightarrow A_\af$ induces a unital positive embedding of $K_i(A_{\{x\}})$ into $K_i(A_\af)$, and an affine homeomorphism from $T(A_\af)$ onto $T(A_{\{x\}})$.
\end{prop}
\begin{proof}
Since $K_i(C(\Omega))$ is torsion free, we have, from the K$\ddot{\rm u}$nneth Theorem, that $K_i(C(X))\cong C(K,K_i(C(\Omega)))$ for $i=0,1$. Since $c_x={\rm id}$ for all $x\in X$, we know that $\af$ induces the action $\af_{*i}: f\mapsto f\circ\taf$ on $K_i(C(X))$. Then according to the Pimsner-Voiculescu six term exact sequence and the minimality of $\taf$, we have
\[0\longrightarrow K^0(K,\taf)\otimes K_i(C(\Omega))\longrightarrow K_i(A_\af)\longrightarrow K_{i+1}(C(\Omega))\longrightarrow 0\]
for $i=0,1$. On the other hand, since $\Omega$ is a finite CW-complex, $K_i(C(\Omega))\,(i=0,1)$ is finitely-generated, and hence the exact sequence splits, which follows the $K$-theory of $A_\af$. The $K$-theory of $A_{\{x\}}$ follows from a similar argument as in (2) of Proposition 3.3 in \cite{LM}, that is, we have 
\[K_i(A_{\{x\}})\cong K^0(K,\taf)\otimes K_i(C(\Omega)).\]
This shows that $K_i(A_\af)\cong K_i(A_{\{x\}})\oplus K_{i+1}(C(\Omega))$ for $i=0,1$ and the embedding from $K_i(A_{\{x\}})$ into $K_i(A_\af)$ which sending $x$ to $(x,0)$ is induced by the canonical embedding $j_x: A_{\{x\}}\hookrightarrow A_\af$. By checking the Corollary 1.12 of \cite{WS} carefully, one notice that $\Z^2$ can be replaced by any torsion free abelian group $G$, and hence $K_i(A_{\{x\}})\cong K^0(K,\taf)\otimes K_i(C(\Omega))$ is torsion free for $i=0,1$. Finally, according to (4) of Proposition 3.3 in \cite{LM}, $j_x: A_{\{x\}}\hookrightarrow A_\af$ induces an affine homeomorphism from $T(A_\af)$ onto $T(A_{\{x\}})$.
\end{proof}
\begin{df}\label{3.5}
Let $Y$ be a compact metrizable space and $\bt:Y\to Y$ be a minimal homeomorphism. We say that $(Y,\bt)$ is ($C^*$-)rigid, if the image of $K_0(A_\bt)$ under $\rho_{A_\bt}$ is uniformly dense in $\Aff(T(A_\bt))$. We also say $\bt$ is rigid, if $(Y,\bt)$ is a rigid minimal dynamical system.
\end{df}
Note that according to \cite{LP}, we have
\begin{prop}\label{3.6}
A minimal dynamical system $(Y,\bt)$ is rigid if and only if 
\[TR(A_\bt)=0.\]
\end{prop}
\begin{df}[Definition 3.6, \cite{LM}]
A unital simple $C^*$-algebra $A$ is called an almost tracially $AF$-algebra, if there exists a unital simple subalgebra $B\subset A$ with $TR(B)=0$ such that the following holds: for any finite subset $\F\subset A$ and any $\ep>0$, there exists a projection $p\in B$ such that

(1) For every $a\in\F$, there exists $b\in B$ with $\|ap-b\|<\ep$;

(2) $\tau(1-p)<\ep$ for every $\tau\in T(B)$.
\end{df}
\begin{thm}\label{3.7}
$RR(A_{\{x\}})=0$ for some $x\in K$ implies the rigidity of $\af$.
\end{thm}
\begin{proof}
Let $x\in K$ with $RR(A_{\{x\}})=0$. By Theorem \ref{3.1}, $A_{\{x\}}$ is a simple unital AH-algebra with no dimension growth. It follows from the Proposition 2.6 of \cite{L2} that $TR(A_{\{x\}})=0$. We now show that $A_\af$ is an almost tracially AF-algebra. In fact, let $\ep>0$ and 
\[\F\subset\left\{\sum_{k=-N}^Nf_ku_\af^k: f_k\in C(X)\right\}=E_N\]
 be a finite subset. Since $\af$ is minimal, we can find a clopen set $U\ni x$ such that $\af^{1-N}(U), \af^{2-N}(U), \cdots, U, \cdots, \af^N(U)$ are mutually disjoint and $\mu(U)<\ep/2N$ for all Borel probability measure $\mu$ on $X$. Let $V=\bigcup_{k=1-N}^N\af^k(U)$ and $p=1_{(K\setminus V)\times\Omega}$. It is clear that $u_\af^kp, u_\af^{1-k}p\in A_{\{x\}}$ for $k=0,1,\cdots,N$ and hence $ap\in A_{\{x\}}$ for every $a\in E_N$. Besides, $\tau(1-p)<\ep$ for every $\tau\in T(A_\af)=T(A_{\{x\}})$ as is shown in Proposition \ref{3.3}. This verifies that $A_\af$ is an (simple) almost tracially AF-algebra, and therefore $RR(A_\af)=0$ by Theorem 3.10 in \cite{LM}.

On the other hand, it follows from \cite{LN} that $A_\af$ has stable rank one. Consequently, $A_\af$ is stably finite and the projections in $A_\af\otimes\mathcal{K}$ satisfy cancellation. Furthermore, $K_0(A_\af)$ is unperforated for the strict order, according to Theorem 4.5 of \cite{P}. This implies that the image of $\rho_{A_\af}: K_0(A_\af)\to \Aff(QT(A_\af))$ is uniformly dense. Since quasitraces on nuclear $C^*$-algebras exactly are traces, $QT(A_\af)=T(A_\af)$ and hence $\af$ is rigid, which completes the proof.
\end{proof}
\begin{rem}\label{3.8}
1. According to Theorem 10.6 of \cite{Mun} and Chapter IV, 12 of \cite{Whi}, if $\Omega$ is a smooth manifold, then $\Omega$ admits a simplicial structure, and hence is a finite CW-complex. Additionally, Theorem 4.2 of \cite{FNOP} implies that, for a finite-dimensional connected compact manifold $M$, if ${\rm dim}\,M\leq 3$ or if ${\rm dim}\,M\leq 5$ and $M$ is closed, then $M$ also admits the structure of a finite CW-complex.

2. In \cite{LM}, it is shown that when $\Omega=\T$, then there is an equivalent between the following four conditions:

(1) The minimal homeomorphism $\af$ is rigid;

(2) The canonical projection $\pi: K\times\T\to K$ induces an affine homeomorphism from $M_\af(X)$ onto $M_{\taf}(K)$, the spaces of invariant probability measures;

(3) $RR(A_{\{x\}})=0$;

(4) $RR(A_\af)=0$.

In particular, the implication $(2)\Rightarrow (3)$ follows from the fact that, with the $K$-theory of $A_{\{x\}}$, there is an equivalence between the conditions that $\pi$ induces an affine homeomorphisms of measure spaces, and that the projections in $A_{\{x\}}$ separate traces. Consequently, the condition that ${\rm dim}(K\times\T)=1$ implies $RR(A_{\{x\}})=0$, by invoking Theorem 1.3 of \cite{BBEK}. However, for a finite CW-complex, we only have the following implications, according to Theorem \ref{3.7}:
\[RR(A_{\{x\}})=0\Rightarrow (X,\af)\ {\rm is\ rigid}\Leftrightarrow RR(A_\af)=0.\]
\end{rem}

\vspace{0.3cm}

\section{Lipschitz-minimal-property and minimal skew products}\label{sec:4}
In this section, we keep the assumption that $K$ is the Cantor space and $\Omega$ admits the structure of an (infinite) finite-dimensional connected finite CW-complex. Denote $X=K\times\Omega$. Start with a Cantor minimal system $(K,\taf)$, we apply an argument similar to the proof of Theorem 1 in \cite{GW} and construct a class $\mathcal{R}_0(\taf)$ of minimal skew products on $X$, whose elements will be classified up to approximate $K$-conjugacy later in Section 6.
\begin{df}\label{4.1}
Let $Z$ be a compact metric space and $M\subset \Homeo(Z)$ be a closed subset. We shall say $M$ acts minimally on $Z$, if for every nonempty open set $O\subset Z$, there is a finite set $\{h_1,h_2,\cdots,h_n\}\subset M$ such that $\bigcup_{1\leq i\leq n}h_i(O)=Z$.
\end{df}
Recall that for a compact metric space $Z$, a homeomorphism $f\in\Homeo(Z)$ and $\Lip>0$, we say $f$ is $\Lip$-Lipschitz, if $d(f(z_1),f(z_2))\leq\Lip d(z_1,z_2)$ for all $z_1, z_2\in Z$. 
\begin{df}\label{4.2}
For a compact metric space $Z$ and $\Lip>0$, we denote the set of $\Lip$-Lipschitz homeomorphisms in the identity path component of $\Homeo(Z)$ by $\G_{\Lip}(Z)$. We shall say $Z$ satisfies the {\em Lipschitz-minimal-property}, if there exists  $\Lip>0$ and a pathwise connected set $M\subset\G_{\Lip}(Z)$ which acts minimally on $\Omega$. 
\end{df}
From now on, for a compact metric space which satisfies the Lipschitz-minimal-property, we will simply say $Z$ satisfies the LMP. Note that from Definition \ref{4.2}, it is straightforward to verify the following proposition.
\begin{prop}\label{4.3}
If $Z_1$ and $Z_2$ satisfies the LMP, then $Z_1\times Z_2$ satisfies the LMP.
\end{prop}
\begin{rem}\label{4.4}
It is clear that for all $n\geq1$, the $n$-torus $\T^n$ and the $n$-dimensional sphere $\Sph^n$ satisfy the LMP. In fact, for an $n$-torus, the set of products of rotations on $\T$ and the set of standard rotations on $\Sph^n$ both act freely (on $\T^n$ and $\Sph^n$). Then Proposition \ref{4.3} follows that any finite products of $\T^n\,(n\geq1)$ and $\Sph^m\,(m\geq1)$ satisfies the LMP.
\end{rem}
If $\Omega$ satisfies the LMP, we use $G_{\Lip}(\Omega)$ to denote the closed subgroup of $\Homeo(\Omega)$ generated by $\G_{\Lip}(\Omega)$. Note that in the proof of Theorem 1 in \cite{GW}, it is not necessary to assume that $\G$ is a subgroup which makes $(Y,\G)$ a minimal flow. Actually, it suffices to assume that $\G$ is a connected set and acts minimally on $Y$. Therefore, one could regard the next theorem as a corollary of Theorem 1 in \cite{GW}, which, however, plays a crucial role in this paper and we decide to state the (adapted) proof for the reader's convenience.
\begin{thm}\label{4.5}
Let $\Omega$ satisfy the LMP. For every Cantor minimal system $(K,\taf)$, there exists $\Lip>0$ and uncountably many continuous maps $c\in C(K,G_\Lip)$ such that $(X,\taf\times c)$ is a minimal dynamical system.
\end{thm}
\begin{proof}
According to Definition \ref{4.2}, let $M\subset\G_{\Lip}$ be a pathwise connected set which acts minimally on $\Omega$ for some $\Lip>0$. For every $g\in C(K,M)$, use $\id\times g$ to denote the homeomorphism $(x,\omega)\mapsto (x,g_x(\omega))$. Define
\[\mathcal{D}_0(\taf)=\{G^{-1}\circ\taf\circ G: G=\id\times g\ {\rm for\ some\ }g\in C(K,M)\}\]
where we identify $\taf$ with $\taf\times\id$. Let
\[\Raf=\{\taf\times c\in \overline{\mathcal{D}_0(\taf)}: \taf\times c\ {\rm is\ minimal}\}\]
We will show that $\overline{\Raf}=\overline{\mathcal{D}_0(\taf)}$. Let $O$ be a nonempty open set of $X$. Denote
\[E_O=\left\{h\in\overline{\mathcal{D}_0(\taf)}:\bigcup_{i=0}^\infty h^i(O)=X\right\}\]
The compactness of $X$ implies that $E_O$ is open in $\overline{\mathcal{D}_0(\taf)}$ and it is clear the $\bigcap_{i}E_{O_i}$ is precisely the set of minimal homeomorphisms in $\overline{\mathcal{D}_0(\taf)}$, if $O_i$ is a countable basis of $X$. Since $\overline{\mathcal{D}_0(\taf)}$ is complete, according to the Baire's category theorem, it is equivalent to show that $E_O$ is dense in $\overline{\mathcal{D}_0(\taf)}$. It is straightforward to check that for $G=\id\times g$ with $g\in C(K,M)$, we have $GE_OG^{-1}=E_{G(O)}$, which follows that $G\overline{E_O}G^{-1}=\overline{E_{G(O)}}$.  Note that to prove $E_O$ is dense in $\overline{\mathcal{D}_0(\taf)}$, it is equivalent to prove that $G^{-1}\circ\taf\circ G\in \overline{E_O}$ for all $G=\id\times g$ with $g\in C(K,M)$, or in other words, $\taf\in\overline{E_{G^{-1}(O)}}$. Therefore, it suffices to show that $\taf\in \overline{E_O}$ for all open set $U$, that is, for every $\ep>0$, there is $G=\id\times g$ with $g\in C(K,M)$ such that
\[d(\taf, G^{-1}\circ\taf\circ G)<\ep\ \ {\rm and}\ \ G^{-1}\circ\taf\circ G\in E_O\]
Now fix $\ep>0$. Let $W\subset K$ and $V\subset \Omega$ be nonempty open sets such that $W\times V\subset O$. Since $M$ acts minimally on $\Omega$, there is a finite set $\{h_0,h_{1/n},\cdots,h_{1-1/n}\}\subset M$ with $\bigcup_{0\leq i\leq n-1} h_{i/n}(V)=\Omega$. Note that $M$ is pathwise connected, which follows that there is a continuous extension $\tilde{h}$ from $I=[0,1]$ to $M$ with
\[\tilde{h}_{i/n}=h_{i/n}\ (i=0,1,\cdots,n-1)\]
By abuse of notation, we still denote the extension by $h$. Choose $\delta>0$ such that $d(h_{t_1}^{-1}\circ h_{t_2}, \id)<\ep$ whenever $|t_1-t_2|<\delta$. Note that since $h$ takes values in $M\subset \G_{\Lip}$, $h_{t_1}^{-1}\circ h_{t_2}\in G_{\Lip}$. Take $m\geq1$ with $2/m<\delta$ and choose an open set $A\subset W$ such that 
\[A,\taf(A),\cdots,\taf^{m-1}(A)\]
are pairwise disjoint. Note that this is feasible since $\taf$ is minimal. Let $J\subset A$ be a Cantor set and take a continuous surjection $\tilde{\theta}:J\to [0,1]$. For $x\in\taf^i(J)$, define
\[\tilde{\theta}(x)=\tilde{\theta}(\taf^{-i}(x))\]
which makes $\tilde{\theta}$ a continuous map of $\bigcup_{0\leq i\leq m-1}\taf^i(J)\subset K$ onto $[0,1]$. Now we extend $\tilde{\theta}$ to a continuous map $\tilde{\theta}:K\to[0,1]$. Finally, define $\theta:K\to [0,1]$ by
\[\theta(x)=\frac{1}{m}\sum_{i=0}^{m-1}\tilde{\theta}(\taf^i(x))\]
It is immediate that $\theta$ is an onto map and $\theta|_J=\tilde{\theta}|_J$. Now we define $g:K\to M$ by $g_x=h_{\theta(x)}$ and put $G(x,\omega)=(x,g_x(\omega))$. The verification that the homeomorphism $G$ satisfies $d(\taf, G^{-1}\circ\taf\circ G)<\ep$ and $G^{-1}\circ\taf\circ G\in E_O$ is exactly the same as the final part of Theorem 1 in \cite{GW} and we omit this part.

Finally, note that the argument above shows that $\Raf$ is a countable intersection of dense open sets in a complete metric space, and hence it is a dense $G_\delta$ set and a residual set. It is easy to see that $\overline{\mathcal{D}_0(\taf)}$ has no isolated point, and therefore $\Raf$ has no isolated point too, which follows that $\Raf$ is an uncountable set by the Baire's category theorem.  This completes the proof.
\end{proof}
\begin{df}
Throughout this paper, for a Cantor minimal system $(K,\taf)$ and an infinite finite-dimensional connected finite CW-complex $\Omega$ which satisfies the LMP, we use $\Raf$ to denote the class of minimal homeomorphisms in Theorem \ref{4.5}.
\end{df}
\begin{rem}\label{4.7}
According to Theorem \ref{4.5}, we know that for every $\af=\taf\times c\in\Raf$ and every $\ep>0$, there exists $\Lip>0$ and $g\in C(K,\G_{\Lip})$ such that
\[d(c_x(\omega),g_{\taf(x)}^{-1}\circ g_x(\omega))<\ep\]
for all $x\in K$ and $\omega\in\Omega$. In particular, since $c\in C(K,G_{\Lip})$ and $G_{\Lip}$ is contained in the identity component of $\Homeo(\Omega)$, then $c_x$ induces the identity map on $K_i(C(\Omega))$ for all $x\in K$ and $i=0,1$.
\end{rem}
\begin{lem}\label{4.8}
Let $\Omega$ satisfie the LMP and $(K,\taf)$ be a Cantor minimal system. Then there exists $\Lip>0$ such that for every $\af=\taf\times c\in\Raf$ and $\ep>0$, there exists $h\in C(K,G_{\Lip})$ such that
\[d(h_{\taf(x)}^{-1}\circ c_x\circ h_x(\omega),\omega)<\ep\]
for all $x\in K$ and $\omega\in\Omega$.
\end{lem}
\begin{proof}
Let $\Lip>0$ be associated with the LMP of $\Omega$ and fix $\ep>0$. According to Remark \ref{4.7}, there is $g\in C(K,\G_{\Lip})$ such that $d(c_x(\omega),g_{\taf(x)}^{-1}\circ g_x(\omega))<\ep/\Lip$ for all $x\in K$ and $\omega\in\Omega$. Upon replacing $\omega$ with $g_x^{-1}(\omega)$, we have
\[d(c_x\circ g_x^{-1}(\omega), g_{\taf(x)}^{-1}(\omega))<\ep/\Lip\]
Note that $g$ takes values in $M\subset \G_{\Lip}$, which follows that $g_x$ is $\Lip$-Lipschitz for all $x\in K$. Therefore, we have
\[d(g_{\taf(x)}\circ c_x\circ g_x^{-1}(\omega),\omega)\leq \Lip d(c_x\circ g_x^{-1}(\omega), g_{\taf(x)}^{-1}(\omega))<\ep\]
By taking $h_k=g_k^{-1}$, we complete the proof.
\end{proof}
\begin{rem}\label{4.9}
Let $\Omega=\T$ and $(K,\taf)$ be any Cantor minimal system. Let $\xi:K\to\T$ be a continuous map and define a homeomorphism $\taf\times R_{\xi}$ of $X=K\times\T$ by $(x,t)\mapsto (\taf(x), t+\xi(x))$ as in \cite{LM}. Then $\taf\times R_\xi\in \overline{\mathcal{D}_0(\taf)}$. In fact, according to Lemma 6.1 of \cite{LM}, for any $\ep>0$, there is $\eta\in C(K,\T)$ such that 
\[d(\xi(x),\eta(x)-\eta(\taf(x)))<\ep\]
for all $x\in K$. Consequently, if $\taf\times R_\xi$ is minimal, then $\taf\times R_\xi\in\Raf$. Similarly, for $\Omega=\T^2$, all the minimal homeomorphisms of the form $\taf\times R_\xi\times R_\eta$ considered in \cite{WS}, which send $(x,t_1,t_2)$ to $(\taf(x), t_1+\xi(x),t_2+\eta(x))$ are elements of $\Raf$.
\end{rem}

\vspace{0.2cm}

\section{Approximate conjugacies}\label{sec:5}
The concepts of approximate conjugacies are defined early in the works of the classification of minimal dynamical systems on the Cantor space (see \cite{LM1}), on the product of the Cantor space and the circle (see \cite{LM} and \cite{LM2}) of H. Lin and H. Matui, and later appeared in \cite{WS}, where minimal dynamical systems on the product of the Cantor space and the torus are considered by W. Sun. See also \cite{L3} or \cite{M}  for more details. 

In this section, we keep the condition that $\Omega$ is a (not necessarily infinite) finite-dimensional connected finite CW-complex. Additionally, we assume $K_0(C(\Omega))$ is torsion free and $K_1(C(\Omega))=0$. For a Cantor minimal system $(K,\taf)$, we write $X=K\times \Omega$ and $A_\af=C^*(X,\af)$.
\begin{df}\label{5.1}
Let $Z$ be a compact metric space. Let $(Z,\af)$ and $(Z,\bt)$ be two dynamical systems. We say that $(Z,\af)$ is {\em unilaterally (weakly) approximate conjugate to $(Z,\bt)$ via $\sigma_n$}, if $\sigma_n:Z\to Z$ is a sequence of homeomorphisms of $Z$ such that
\[\lim_{n\to\infty} d(\sigma_n\circ\af\circ\sigma_n^{-1},\bt)=0\]
We say that $(X,\alpha)$ and $(X,\beta)$ are  {\it (bilaterally weakly) approximately conjugate (via $\sigma_n$ and $\gamma_n$)}, if $\sigma_n,\gamma_n:X\to X$ are two sequences of homeomorphisms of $X$ such that 

(1) $(X,\af)$ is unilaterally approximately conjugate to $(X,\bt)$ via $\sigma_n$ and,

(2) $(X,\bt)$ is unilaterally approximately conjugate to $(X,\af)$ via $\gamma_n$

\noindent We write $(X,\alpha)\sim_{{\rm app}}(X,\beta)$ (via $\sigma_n,\gamma_n$), if $(X,\alpha)$ and $(X,\beta)$ are approximately conjugate (via $\sigma_n$ and $\gamma_n$).
\end{df}
\begin{rem}\label{5.2}
By the unilaterally approximate conjugacy of $(X,\af)$ to $(X,\bt)$ via $\sigma_n$, we just mean the following approximately (or almost) commutative diagram:
\begin{equation}
\begin{tikzcd}
Z   \ar[r, "\af"]  \ar[d, rightsquigarrow, "\sigma_n"'] \ar[dr, phantom, "\circlearrowleft_{\varepsilon_n}"] & Z \ar[d, rightsquigarrow, "\sigma_n"]\\
Z \ar[r, rightarrow, "\bt"] & Z
\end{tikzcd}
\end{equation}
Actually, compared with the topological conjugacy (or flip conjugacy), the weakly approximate conjugacy is a rather weak condition. For example, Remark \ref{4.7} and Lemma \ref{4.8} imply the following lemma, which shows that a minimal dynamical system can even be bilaterally approximately conjugate to a non-minimal system with uncountable minimal subsets.
\end{rem}
\begin{lem}\label{5.3}
Let $(K,\taf)$ be a Cantor minimal system and $\Omega$ satisfies the LMP. Then for every $\taf\times c\in\Raf$, $(X,\taf\times\id)$ is (bilaterally) weakly approximately conjugate to $(X,\taf\times c)$, via conjugate maps whose cocycles induce the identity maps on $K_i(C(\Omega))\,(i=0,1)$ at every point of $K$.
\end{lem}
\begin{proof}
It is clear from Definition \ref{5.1} that, for dynamical systems $(Z,\af)$ and $(Z,\bt)$, the unilaterally approximate conjugacy of $(Z,\af)$ to $(Z,\af)$ is equivalent to the condition that, for every $\ep>0$, there is $\sigma\in\Homeo(Z)$ such that
\[d(\sigma\circ\af\circ\sigma^{-1},\bt)<\ep\]
According to Corollary \ref{4.7} and Lemma \ref{4.8}, it is straightforward to check that $d(\taf\times c, G^{-1}\circ( \taf\times\id)\circ G)<\ep/\Lip$ and $d(H^{-1}\circ(\taf\times c)\circ H, \taf\times\id)<\varepsilon$, where $G=\id\times g$ and $H=\id\times g^{-1}$ constructed in Lemma \ref{4.8}.
\end{proof}
The next proposition is easy to check, however, will be used in the next section, so we state it and omit the proof.
\begin{prop}\label{5.4}
$\cdot\sim_{\rm app}\cdot$ is an equivalent relation. In particular, if $(Z,\af)\sim_{\rm app}(Z,\bt)$ via $\zeta_n$ and $(Z,\bt)\sim_{\rm app}(Z,\lambda)$ via $\nu_n$,  then for every $\ep>0$, there exist integers $N_1,N_2\ge1$ such that 
\[d(\nu_{N_1}\circ\zeta_{N_2}\circ\af\circ\zeta_{N_2}^{-1}\circ\nu_{N_1}^{-1},\lambda)<\ep\]
that is, $(Z,\af)\sim_{\rm app}(Z,\lambda)$ via a subset of $\{\nu_m\circ\zeta_n: m,n\geq1\}$.
\end{prop}
We now introduce the following lemma, which will play a leading role in defining the approximate $K$-conjugacy of dynamical systems $(Z,\af)$ and $(Z,\bt)$.
\begin{lem}[Proposition 3.2, \cite{LM1}]\label{5.5}
Let $Z$ be a compact metric space and let $\af$ and $\bt$ be minimal homeomorphisms on $Z$. If there exists homeomorphisms $\sigma_n:Z\to Z$ such that $(Z,\af)$ is unilaterally approximately conjugate to $(Z,\bt)$, then there exists a unital asymptotic morphism $\{\psi_n\}$ from $A_\bt$ to $A_\af$ such that for all $f\in C(Z)$,
\[\lim_{n\to\infty}\|\psi_n(u_\bt)-u_\af\|=0\]
and
\[\lim_{n\to\infty}\|\psi_n(j_\bt(f))-j_\af(f\circ\sigma_n)\|=0.\]
\end{lem}

\begin{df}\label{5.6}
Let $Z$ be a compact metric space and $\alpha,\beta:Z\to Z$ be two minimal rigid homeomorphisms. We say that $(Z,\alpha)$ and $(Z,\beta)$ are ({\it weakly}){\it approximately $K$-conjugate}, if $(Z,\alpha)\sim_{\rm app}(Z,\beta)$ via $\sigma_n,\gamma_n$ for some $\sigma_n, \gamma_n\in \mathcal{H}(Z)$ in the sense of Definition \ref{5.1}, and there exists sequences of unital asymptotic morphisms $\varphi_n, \psi_n$ such that 
\[\lim_{n\to\infty}\|\psi_n(j_\bt(f))-j_\af(f\circ\sigma_n)\|=0, \ \ \ \lim_{n\to\infty}\|\phi_n(j_\af(f))-j_\bt(f\circ\gamma_n)\|=0\]
for all $f\in C(Z)$. Besides, there exists 
$\kappa\in KL(A_\alpha, A_\beta)$ which induces an isomorphism
\[\tilde{\kappa}: (K_0(A_\alpha), K_0(A_\alpha)_+, [1_{A_\alpha}], K_1(A_\alpha), T(A_\alpha))\to (K_0(A_\beta), K_0(A_\beta)_+, [1_{A_\beta}], K_1(A_\beta), T(A_\beta))\]
and with $[\phi_n]=[\phi_1]=[\psi_n]^{-1}=\tilde{\kappa}$ for all $n\geq1$.
\end{df}
\begin{df}\label{5.7}
Let $Z$ be a compact metric space and $\alpha,\beta:Z\to Z$ be two rigid minimal homeomorphisms. We say that $(Z,\alpha)$ and $(Z,\beta)$ are {\it $C^*$-strongly approximately conjugate},  if there exits sequences of isomorphisms $\varphi_n:A_\alpha\to A_\beta, \psi_n:A_\beta\to A_\alpha$ and isomorphisms $\chi_n, \lambda_n: C(Z)\to C(Z)$ such that $[\varphi_n]=[\varphi_1]=[\psi_n^{-1}]$ in $KL(A_\alpha,A_\beta)$ for all $n\geq1$ and 
\[\lim_{n\to\infty}\|\varphi_n\circ j_\alpha(f)-j_\beta\circ\chi_n(f)\|=0\ \ {\rm and}\ \ \lim_{n\to\infty}\|\psi_n\circ j_\beta(f)-j_\alpha\circ\lambda_n(f)\|=0\]
for all $f\in C(Z)$.
\end{df}
\begin{rem}\label{5.8}
The Definition \ref{5.6} and Definition \ref{5.7} is only for the case $(Z,\af)$ and $(Z,\bt)$ are rigid, that is, $TR(A_\af)=TR(A_\bt)=0$. The notion of $C^*$-strongly approximate conjugacy inherits the philosophy of Tomiyama's classification theorem for topologically free dynamical systems, which states that $(Z,\af)$ and $(Z,\bt)$ are flip conjugate (that is, either $(Z,\af)$ and $(Z,\bt)$ or $(Z,\af)$ and $(Z,\bt^{-1})$ are topologically conjugate) if and only if there is a unital isomorphism $\phi: A_\af\to A_\bt$ such that the following diagram commutes:
\begin{center}
\begin{tikzcd}
C(Z)   \ar[r, "\phi|_{C(Z)}"]  \ar[d, hookrightarrow, "j_\alpha"'] \ar[dr, phantom, "\circlearrowleft"] & C(Z) \ar[d, hookrightarrow, "j_\beta"]\\
A_\af \ar[r, rightarrow, "\phi"] & A_\beta
\end{tikzcd}
\end{center}
in other words, the isomorphism $\phi$ keeps the commutative subalgebra $C(Z)$.
\end{rem}
\begin{lem}\label{5.9}
Let $\beta\in\Homeo(X)$ be a minimal rigid homeomorphism. Then, for any function $f\in C(X)_+$ and $\varepsilon>0$, there exist projections $p_1,p_2,\cdots,p_k\in \mathcal{P}(A_\beta)$ and real numbers $\lambda_1,\lambda_2,\cdots,\lambda_k$ such that, there are $x_i=[q_i]-[q_i']\in K_0(C(X))$ with $(j_\beta)_{*0}(x_i)=[p_i]$ for $i=1,2,\cdots,k$ and 
\[\textstyle|\rho_{A_\beta}(j_\beta(f))(\tau)-\sum_{i=1}^k\lambda_i\rho_{A_\beta}(j_\beta(y_i))(\tau)|<\varepsilon\]
for all $\tau\in T(A_\beta)$, where $y_i=q_i-q_i'\in M_{\infty}(C(X))$.
\end{lem}
\begin{proof}
Fix $f\in C(X)_+$ and $\ep>0$. Since $(X,\bt)$ is rigid, $TR(A_\bt)=0$ by Proposition \ref{3.6}, and hence $RR(A_\bt)=0$. It follows by Theorem 3.25 of \cite{L1} that the set of elements in $(A_\bt)_{s.a.}$ with finite spectrum is dense in $(A_\bt)_{s.a.}$, that is, there are mutually orthogonal projections $\{p_1,p_2,\cdots,p_k\}\subset\Pj(A_\bt)$ and real numbers $\{\lambda_1,\lambda_2,\cdots,\lambda_k\}$ such that
\[\textstyle\|j_\bt(f)-\sum_{i=1}^k \lambda_i p_i\|<\ep\]
Since $K_1(C(\Omega))=0$ and $K_0(C(\Omega))$ is torsion free, we have $K_1(C(X))=0$ and the following exact sequence induced by the Pimsner-Voiculescu six term exact sequence
\[0\longrightarrow K_0(C(X))/{\rm ker}(\id-\bt_*)\longrightarrow K_0(C(X))\stackrel{(j_\bt)_{*0}}{\longrightarrow} K_0(A_\bt)\rightarrow 0,\]
which follows that $(j_\bt)_{*0}$ is surjective. Therefore, we can choose a finite set 
\[\F=\{x_i=[q_i]-[q_i']: q_i\in \Pj_{\infty}(C(X)), 1\leq i\leq k\}\subset K_0(C(X))\]
such that $(j_\bt)_{*0}(x_i)=[p_i]$ in $K_0(A_\bt)$ for $i=1,2,\cdots,k$. Note that equivalent projections have same traces, which follows $\rho_{A_\bt}(p_i)(\tau)=\rho_{A_\bt}(j_\bt(q_i-q_i'))(\tau)$ for all $\tau\in T(A_\bt)$ and $i=1,2,\cdots,k$. Let $y_i=q_i-q_i'$ and write
$z=\sum_{i=1}^k\lambda_i p_i$. Then we have
$\rho_{A_\bt}(z)(\tau)=\sum_{i=1}^k\lambda_i\rho_{A_\bt}(j_\bt(y_i))(\tau)$ for all $\tau\in T(A_\bt)$, which implies that
\begin{align*}
&\textstyle|\rho_{A_\beta}(j_\beta(f))(\tau)-\sum_{i=1}^k\lambda_i\rho_{A_\beta}(j_
\beta(y_i))(\tau)|\\
\leq&\textstyle|\rho_{A_\beta}(j_\beta(f))(\tau)-\rho_{A_\beta}(z)(\tau)|+|\rho_{A_\beta}(z)(\tau)-\sum_{i=1}^k\lambda_i\rho_{A_\beta}(j_
\beta(y_i))(\tau)|\\
\leq&\|j_\beta(f)-z\|<\varepsilon,
\end{align*}
This completes the proof.
\end{proof}
\begin{lem}\label{5.10}
Let $\beta\in\Homeo(X)$ be minimal and rigid.  Let $h: C(X)\hookrightarrow A_\beta$ be a unital monomorphism and $h_n: C(X)\hookrightarrow A_\beta$ be a sequence of unital monomorphisms. Suppose that for any finite subset $F\subset K^0(X)$, there is $N\geq1$, such that 
\[\rho_{A_\beta}(h_{*0}(x))(\tau)=\rho_{A_\beta}({h_n}_{*0}(x))(\tau)\]
for all $n\geq N$, $\tau\in T(A_\beta)$ and $x\in F$, then 
\[\rho_{A_\beta}(h(f))(\tau)=\lim_{n\to\infty}\rho_{A_\beta}(h_n(f))(\tau)\]
for all $f\in C(X)$ and $\tau\in T(A_\beta)$.
\end{lem}
\begin{proof}
Let $f\in C(X)$, $\tau\in T(A_\beta)$ and $\varepsilon>0$. Without loss of generality, we assume $f\geq0$. 
Then by Lemma \ref{5.9}, there is an integer $k\geq1$, projections $p_1,p_2,\cdots,p_k\in \mathcal{P}(A_\beta)$ and real numbers $\lambda_1,\lambda_2,\cdots,\lambda_k$ such that there exist $x_i=[q_i]-[q_i']\in K_0(C(X))$ with $(j_\beta)_{*0}(x_i)=[p_i]$ for $i=1,2\cdots,k$ and
\begin{align}\label{5.5}
\textstyle|\rho_{A_\beta}(j_\beta(f))(\tau)-\sum_{i=1}^k\lambda_i\rho_{A_\beta}(j_\beta(y_i))(\tau)|<\varepsilon/2
\end{align}
for all $\tau\in T(A_\beta)$, where $y_i=q_i-q_i'$, $i=1,2,\cdots,k$.  For every $n\geq1$, we define traces $\tilde{\tau}$ and $\tilde{\tau}_n$ on $C(X)$ by 
\[\tilde{\tau}(g)=\tau(h(g))\ \ {\rm and}\ \ \tilde{\tau}_n(g)=\tau(h_n(g))\]
It is straightforward to check that $\tilde{\tau}(g\circ\beta^{-1})=\tilde{\tau}(g)$ and $\tilde{\tau}_n(g\circ\beta^{-1})=\tilde{\tau}_n(g)$
for all $g\in C(X)$. Hence, $\tilde{\tau}$ and $\tilde{\tau}_n$ can be extended to $T(A_\beta)$, which we still denote by $\tilde{\tau}$ and $\tilde{\tau}_n$.  According to \eqref{5.5}, we have
\begin{align*}
\begin{cases}
\vspace{0.1cm}
\textstyle|\rho_{A_\beta}(j_\beta(f))(\tilde{\tau})-\sum_{i=1}^k\lambda_i\rho_{A_\beta}(j_\beta(y_i))(\tilde{\tau})|=|\rho_{A_\beta}(h(f))(\tau)-\sum_{i=1}^k\lambda_i\rho_{A_\beta}(h(y_i))(\tau)|<\varepsilon/2\\
\textstyle|\rho_{A_\beta}(j_\beta(f))(\tilde{\tau}_n)-\sum_{i=1}^k\lambda_i\rho_{A_\beta}(j_\beta(y_i))(\tilde{\tau}_n)|=|\rho_{A_\beta}(h_n(f))(\tau)-\sum_{i=1}^k\lambda_i\rho_{A_\beta}(h_n(y_i))(\tau)|<\varepsilon/2
\end{cases}
\end{align*}
Let $F=\{x_1,x_2,\cdots,x_k\}\subset K_0(C(X))$. According to the assumption, there is an integer $N\geq1$, such that for all $n\geq N$ and $i=1,2,\cdots,k$, 
\[\rho_{A_\beta}(h(y_i))(\tau)=\rho_{A_\beta}(h_n(y_i))(\tau)\]
This follows $\sum_{i=1}^k\lambda_i\rho_{A_\beta}(h(y_i))(\tau)=\sum_{i=1}^k\lambda_i\rho_{A_\beta}(h_n(y_i))(\tau)$ for all $n\geq N$ and finally,
\begin{align*}
&|\rho_{A_\beta}(h(f))(\tau)-\rho_{A_\beta}(h_n(f))(\tau)|\\
\leq \ &\textstyle|\rho_{A_\beta}(h(f))(\tau)-\sum_{i=1}^k\lambda_i\rho_{A_\beta}(h(y_i))(\tau)|\\
+\ &\textstyle|\sum_{i=1}^k\lambda_i\rho_{A_\beta}(h_n(y_i))(\tau)-\rho_{A_\beta}(h_n(f))(\tau)|\\
<\ &\varepsilon/2+\varepsilon/2=\varepsilon.
\end{align*}
whenever $n\geq N$. This completes the proof.
\end{proof}
\begin{lem}\label{5.11}
Let $\beta\in\Homeo(X)$ be minimal and rigid. Let $h_1, h_2: C(X)\hookrightarrow A_\beta$ be two unital monomorphisms. Suppose that for any $x\in K_0(C(X))$,
\[\rho_{A_\bt}\circ(h_1)_{*0}(x)(\tau)=\rho_{A_\bt}\circ (h_2)_{*0}(x)(\tau)\]
for all $\tau\in T(A_\bt)$, then
\[\rho_{A_\bt}(h_1(f))(\tau)=\rho_{A_{\bt}}(h_2(f))(\tau)\]
for all $f\in C(X)$ and $\tau\in T(A_\bt)$.
\end{lem}
\begin{proof}
By setting $h_n=h_2\,(n\geq1)$ in Lemma \ref{5.10}, we get the conclusion.
\end{proof}
Recall that for two $C^*$-algebras, two maps $h_1,h_2:A\to B$, a finite subset $\F\subset A$ and $\ep>0$, we write $h_1\approx_{\ep} h_2$ on $\F$, if $\|h_1(a)-h_2(a)\|<\ep$ for all $a\in \F$. Furthermore, if $B$ is unital and there is a unitary $u\in B$ such that $\|{\rm ad}_u\circ h_1(a)-h_2(a)\|<\ep$ for all $a\in\F$, then we write $h_1\stackrel{u}{\sim}_\ep h_2$ on $\F$. For our purpose, we need the following theorem.
\begin{lem}[Theorem 3.3,\cite{L3}]\label{5.12}
Let $X$ be a compact metric space and $A$ be a unital separable simple $C^*$-algebra with $TR(A)=0$. Suppose that $h_1: C(X)\hookrightarrow A$ is a unital monomorphism. For any $\varepsilon>0$ and any finite subset $\mathcal{F}\subset C(X)$, there exist $\delta>0$ and a finite subset $\mathcal{G}\subset C(X)$ satisfying the following: if $h_2: C(X)\hookrightarrow A$ is another unital monomorphism such that
\[[h_1]=[h_2]\ {\rm in}\ KL(C(X),A)\ \ {\rm and}\ \ |\tau\circ h_1(f)-\tau\circ h_2(f)|<\delta\]
for all $f\in \mathcal{G}$ and $\tau\in T(A)$, then there exists a unitary $u\in A$ such that
\[h_1\stackrel{u}{\sim}_{\varepsilon} h_2\ {\rm on}\ \mathcal{F}.\]
\end{lem}

\begin{thm}\label{5.13}
Let $\alpha,\beta:X\to X$ be two minimal rigid homeomorphisms. Then the followings are equivalent:

\noindent $(1)$ $(X,\alpha)$ and $(X,\beta)$ are approximately $K$-conjugate;

\noindent $(2)$ There exist homeomorphisms $\sigma_n,\gamma_n:X\to X$ such that $(X,\alpha)\sim_{\rm app}(X,\beta)$ via $\sigma_n,\gamma_n$
 and in addition, for any $x\in K_i(C(X))$, there exists $N$ such that
 \begin{align}
 \label{e5.3}(j_{\alpha}\circ\sigma_n^*)_{*i}(x)=(j_{\alpha}\circ\sigma_{m}^*)_{*i}(x), \ \ \ &(j_{\beta}\circ\gamma_n^*)_{*i}(x)=(j_{\beta}\circ\gamma_{m}^*)_{*i}(x)\\
\label{e5.4}(j_\alpha\circ(\gamma_n\circ\sigma_m)^*)_{*i}(x)=(j_\alpha)_{*i}(x),\ \ \ &(j_\beta\circ(\sigma_n\circ\gamma_m)^*)_{*i}(x)=(j_\beta)_{*i}(x)
\end{align}
 for all $m\geq n\geq N$ and for $i=0,1$.
 
 \noindent $(3)$ There exists a unital isomorphism $\theta:A_{\alpha}\to A_\beta$, sequences of unitaries $\{w_n\}\subset A_\alpha$ and $\{u_n\}\subset A_\beta$ and sequences of homeomorphisms $\sigma_n,\gamma_n: X\to X$ such that $(X,\alpha)\sim_{\rm app}(X,\beta)$ via $\sigma_n,\gamma_n$ and in addition, 
 \[\lim_{n\to\infty}\|u_n\theta(j_{\alpha}(f))u_n^*-j_{\beta}(f\circ\gamma_n^{-1})\|=0\]
 and
 \[\lim_{n\to\infty}\|w_n\theta^{-1}(j_{\beta}(f))w_n^*-j_\alpha(f\circ \sigma_n^{-1})\|=0\]
 for all $f\in C(X)$.
\end{thm}
\begin{proof}
$(1)\Rightarrow(2)$: Let $(X,\alpha)$ and $(X,\beta)$ be approximately $K$-conjugate, via homeomorphisms $\sigma_n,\gamma_n:X\to X$ and let $\psi_n:A_\beta\to A_\alpha$ and $\phi_n:A_\alpha\to A_\beta$ be associated unital asymptotic morphisms. According to Definition \ref{5.6}, for every $x\in K^i(X)$, there is $N_1\in \mathbb{N}$ such that 
\[(\psi_n\circ j_\beta)_{*i}(x)=(j_{\alpha}\circ\sigma_n^{*})_{*i}(x)\ \ \ {\rm and}\ \ \ (\phi_n\circ j_\alpha)_{*i}(x)=(j_{\beta}\circ\gamma_n^{*})_{*i}(x)\]
for all $n\geq N_1$. Since $[\phi_n]=[\psi_n]^{-1}=\kappa$ for all $n\geq1$, there is $N_2\in\mathbb{N}$ such that 
\[(\psi_n\circ j_\beta)_{*i}(x)=(\psi_m\circ j_\beta)_{*i}(x)\ \ \ {\rm and}\ \ \ (\phi_n\circ j_\alpha)_{*i}(x)=(\phi_m\circ j_\alpha)_{*i}(x)\]
for all $m\geq n\geq N$. Taking $N=\max\{N_1,N_2\}$ gives the first two equations. The last two equations follows similarly, according to $[\phi_n]=[\psi_n]^{-1}=\kappa$ for all $n\geq 1$.
\vspace{0.1cm}

\noindent $(2)\Rightarrow(3)$: Since $(X,\alpha)\sim_{\rm app}(X,\beta)$ via $\sigma_n, \gamma_n$, by Lemma \ref{5.5}, there exist sequentially asymptotic morphisms $\Psi_n: A_\beta\to A_\alpha$ and $\Phi_n: A_\alpha\to A_\beta$ such that 
\begin{align}
\label{e5.8}\lim_{n\to\infty}\|\Psi_n(j_\beta(f))-j_{\alpha}(f\circ\sigma_n^{-1})\|=0, \ \ \ \ \ &\lim_{n\to\infty}\|\Phi_n(j_\alpha(f))-j_{\beta}(f\circ\gamma_n^{-1})\|=0,\\
\label{e5.9}\lim_{n\to\infty}\|\Psi_n(u_\beta)-u_\alpha\|=0,\ \ \ \ \ &\lim_{n\to\infty}\|\Phi_n(u_\alpha)-u_\beta\|=0
\end{align}
for all $f\in C(X)$. This follows that, for any finitely generated subgroup $F_i\subset K_i(A_\alpha)$ and $G_i\subset K_i(A_\beta)$, there is $N\geq1$ such that $[\Phi_n]|_{F_i}$ and $[\Psi_n]|_{G_i}$ are well-defined as homomorphisms for all $n\geq N$ and $i=0,1$. Also note that, for any $(j_\alpha)_{*i}(x)\in (j_\alpha)_{*i}(K^i(X))$ and $(j_\beta)_{*i}(x)\in (j_\beta)_{*i}(K^i(X))$, according to \eqref{e5.8} and \eqref{e5.9}, we have
\begin{align*}
[\Phi_n]((j_\alpha)_{*i}(x))=(j_\beta\circ\gamma_n^*)_{*i}(x),\ \ \ \ \ &[\Psi_n]((j_\beta)_{*i}(x))=(j_\alpha\circ\sigma_n^*)_{*i}(x)\\
[\Phi_n]([u_\alpha])=[u_\beta],\ \ \ \ \ &[\Psi_n]([u_\beta])=[u_\alpha]
\end{align*}
Combining with \eqref{e5.3} and \eqref{e5.4}, there is an isomorphism 
\[\tilde{\kappa}: (K_0(A_\alpha), K_0(A_\alpha)_+, [1_{A_\alpha}], K_1(A_\alpha))\to  (K_0(A_\beta), K_0(A_\beta)_+, [1_{A_\beta}], K_1(A_\beta))\]
such that $[\Phi_n]|_{F_i}=\tilde{\kappa}|_{F_i}$ and $[\Psi_n]|_{G_i}=\tilde{\kappa}^{-1}|_{G_i}$ for all $n\geq N$. Since $TR(A_\alpha)=TR(A_\beta)=0$, there is a unital isomorphism $\theta:A_\alpha\to A_\beta$ such that $\theta_{*}=\tilde{\kappa}$, and in addition, $\rho_{A_\alpha}(K_0(A_\alpha))$ and $\rho_{A_\beta}(K_0(A_\beta))$ are dense in $Aff(T(A_\alpha))$ and $Aff(T(A_\beta))$, which gives a homeomorphism from $T(A_\beta)$ to $T(A_\alpha)$. Moreover, since $A_\alpha$ and $A_\beta$ have torsion free $K$-groups, there is $\kappa\in KL(A_\alpha,A_\beta)$ which induces $\tilde{\kappa}$. Define a sequence of monomorphisms $\Gamma_n: C(X)\to A_\beta$ by 
\[\Gamma_n(f)=j_\beta\circ\gamma_n^*(f)\]
Although $\Gamma_n\neq \theta\circ j_\alpha$, for any finitely generated subgroup $F_i\subset K_i(C(X))$, one has
\[\tilde{\kappa}_i\circ(j_\alpha)_{*i}|_{F_i}=(\theta\circ j_{\alpha})_{*i}|_{F_i}=(\Gamma_n)_{*i}|_{F_i}\]
for sufficiently large $n$. In particular, for any $\tau\in T(A_\beta)$ and projections $p\in \mathcal{S}\subset F_i$ which generates $F_i$, $\tau(\theta\circ j_\alpha(p))=\tau(\Gamma_n(p))$ for sufficiently large $n$, that is, there is $N\geq1$ such that
\[\rho_{A_\beta}(\theta\circ j_{\alpha}(p))(\tau)=\rho_{A_\beta}(\Gamma_n(p))(\tau)\]
for all $\tau\in T(A_\beta)$ and $n\geq N$.
Since $\theta\circ j_{\alpha}$ and $\Gamma_n$ are monomorphisms from $C(X)$ into $A_\beta$, according to Lemma \ref{5.10}, we have
\[\tau(\theta\circ j_{\alpha}(f))=\lim_{n\to\infty}\tau(\Gamma_n(f))\]
for all $f\in C(X)$ and $\tau\in T(A_\beta)$. The above follows that, for any $\varepsilon>0$ and any finite subset $\mathcal{F}\subset C(X)$, there is $N\geq1$ such that
\[[\theta\circ j_{\alpha}]=[\Gamma_n]\ {\rm in}\ KL(C(X),A_\beta)\ \ {\rm and}\ \ |\tau(\theta\circ j_{\alpha}(f))-\tau(\Gamma_n(f))|<\varepsilon\]
for all $n\geq N$, $\tau\in T(A_\beta)$ and $f\in \mathcal{F}$. Then by Lemma \ref{5.12}, for any $\varepsilon>0$, there exists a unitary $u\in A_\beta$ such that
\[\theta\circ j_{\alpha}\stackrel{u}{\sim}_{\varepsilon}\Gamma_n\ \ {\rm on}\ \mathcal{F}\] 
whenever $n\geq N$. This follows that we can find a sequence of unitaries $u_n\in A_\beta$ such that 
\[\lim_{n\to\infty}\|u_n\theta(j_{\alpha}(f))u_n^*-j_\beta(f\circ\gamma_n^{-1})\|=\lim_{n\to\infty}\|u_n\theta(j_{\alpha}(f))u_n^*-\Gamma_n(f)\|=0\]
for all $f\in C(X)$. Similarly, there exists unitaries $w_n\in A_\alpha$ such that
\[\lim_{n\to\infty}\|w_n\theta^{-1}(j_{\beta}(f))w_n^*-j_\alpha(f\circ\sigma_n^{-1})\|=0\]
This proves $(3)$.

\noindent $(3)\Rightarrow (1)$:
Define sequences of isomorphisms $\varphi_n: A_\alpha\to A_\beta$ and $\psi_n: A_\beta\to A_\alpha$ by 
\[\varphi_n(a)=u_n\theta(a)u_n^*\ \ {\rm and}\ \ \psi_n(b)=w_n\theta^{-1}w_n^*\]
One checks that $[\phi_n]=[\psi_n]^{-1}=\kappa$. Note that homomorphisms are clearly asymptotic morphisms, hence $(X,\alpha)$ and $(Y,\beta)$ are approximately $K$-conjugate.
\end{proof}
\begin{rem}\label{5.14}
1. In fact, in \cite{LM1}, the approximate $K$-conjugacy is defined by the statement (2) in Theorem \ref{5.13} for the Cantor minimal systems (see Definition 5.3 of \cite{LM1}). Therefore, Theorem \ref{5.13} shows that whenever $K_1(C(\Omega))=0$, then Definition 5.3 of \cite{LM1} coincides with Definition \ref{5.6}, which includes the case of $\Omega=\{{\rm pt}\}$ (recall that in this section, $\Omega$ is not assumed to be infinite except Lemma \ref{5.2}). 

2. In \cite{LM} and \cite{LM2}, it is shown that when $\Omega=\T$, then one can choose unital isomorphisms  $\phi_n$ and $\psi_n$ instead of unital asymptotic morphisms in Definition \ref{5.6}. Therefore, the equivalence of (1) and (3) in Theorem \ref{5.13} infers that one can always choose $\phi_n$ and $\psi_n$ to be unital isomorphisms whenever $K_1(C(\Omega))=0$.
\end{rem}
\vspace{0.3cm}

\section{Approximate $K$-conjugacy of minimal skew products}\label{sec:6}
In Theorem \ref{4.5}, we have constructed a class $\Raf$ of minimal homeomorphisms on $K\times\Omega$ for every $\Omega$ with the LMP. In this section, we show the following theorem.
\begin{thm}\label{6.1}
Let $(K,\af)$ and $(K,\bt)$ be two Cantor minimal systems and $\Omega$ be an infinite finite-dimensional connected finite CW-complex with the LMP such that $K_0(C(\Omega))$ is torsion free and $K_1(C(\Omega))=0$. Let $\af=\taf\times c_1\in\Raf$ and $\bt=\tbt\times c_2\in \Rbt$ be two minimal rigid homeomorphisms of $X=K\times\Omega$. Then the following conditions are equivalent.
\vspace{0.1cm}

\noindent $(1)$ $(X,\af)$ and $(X,\bt)$ are approximately $K$-conjugate;
\vspace{0.1cm}

\noindent $(2)$ $(X,\af)$ and $(X,\bt)$ $C^*$-strongly approximately conjugate;
\vspace{0.1cm}

\noindent $(3)$ There exists a homeomorphism $\chi: X\to X$ and a unital order isomorphism 
\[\kappa:  (K_0(A_\alpha), K_0(A_\alpha)_+, [1_{A_\alpha}], K_1(A_\alpha))\to(K_0(A_\bt), K_0(A_\bt)_+, [1_{A_\bt}], K_1(A_\bt))\]
 such that the following diagram commutes:
 \begin{center}
\begin{tikzcd}
K_i(C(X))   \ar[r, "(\chi^*)_{*i}"]  \ar[d, rightarrow, "(j_\alpha)_{*i}"'] \ar[dr, phantom, "\circlearrowleft"] & K_i(C(X)) \ar[d, rightarrow, "(j_\beta)_{*i}"]\\
K_i(A_\alpha) \ar[r, rightarrow, "\kappa_i"] & K_i(A_\beta).
\end{tikzcd}
\end{center}
\end{thm}
Throughout this section, we will always keep the condtion on $\Omega$ as in Theorem \ref{6.1}, that is, $\Omega$ is an infinite finite-dimensional connected finite CW-complex with the LMP, torsion free $K_0(C(\Omega))$ and $K_1(C(\Omega))=0$. Denote $X=K\times\Omega$.
\begin{df}\label{6.2}
A map $c\in C(K,\Homeo(\Omega))$ is said to be {\em quasi-trivial}, if $c_x$ induces identity maps on $K_i(C(\Omega))\,(i=0,1)$ for every $x\in K$. 
\end{df}
Recall that if $\af=\taf\times c$ is a minimal homeomorphism of $X$ and $c$ is quasi-trivial, then $K_i(A_\af)\,(i=0,1)$ is torsion free and 
\[K_0(A_\af)\cong K^0(K,\taf)\otimes K_0(C(\Omega)).\]
For a continuous function $f\in C(K,\Z)$ or in $C(K,K_0(C(\Omega)))$, we use $[f]_{\taf}$ or $[f]_{\af}$ to denote the equivalence class of $f$ in $K^0(K,\taf)$ or in $K_0(K,\taf)\otimes K_0(C(\Omega))$ respectively. Also note that in this case, the positive cone of $K_0(A_\af)$ is given by
\[K_0(A_\af)_+=\{[f]_{\af}: f\in C(K,K_0(C(\Omega))): f(x)\in K_0(C(\Omega))_+\ {\rm for\ all}\ x\in K\}\]

Let $\af=\taf\times c_1$ and $\bt=\tbt\times c_2$ be minimal homeomorphisms on $X$ with quasi-trivial cocycles such that $(X, \af)$ and $(X, \bt)$ are $C^*$-strongly approximately conjugate.
\begin{lem}\label{6.3}
Let $O\subset K$ be a clopen set and $f={\bf 1}_{K_0(C(\Omega))}|_O\in C(K,K_0(C(\Omega)))$ which takes the value ${\bf 1}_{K^0(\Omega)}$ in $O$ and ${\bf 0}$ otherwise. Then there is clopen set $W\subset K$ such that
\[\kappa_0([f]_\af)=[{\bf 1}_{K_0(C(\Omega))}|_W]_\bt\]
where $\kappa_0:K_0(A_\af)\to K_0(A_\bt)$ is the unital order isomorphism associated with the isomorphism $\phi_n: A_\af\to A_\bt$ in Definition \ref{5.7}.
\end{lem}
\begin{proof}
Let $\sigma_n,\gamma_n: X\to X$ be the homeomorphisms and $\phi_n: A_\af\to A_\bt, \psi_n: A_\bt\to A_\af$ be the unital isomorphisms associated with the $C^*$-strongly approximate conjugacy of $(X,\af)$ and $(X,\bt)$. According to the Definition \ref{5.7}, for every $f\in K_0(C(X))$, there is $N\geq1$ such that 
\[\kappa_0((j_\af)_{*0}(f))=(j_\bt)_{*0}(f\circ\gamma_n^{-1})\]
for all $n\geq N$. Write $\gamma_n=\tilde{\gamma}_n\times z_n$ for some $z_n\in C(K,\Homeo(\Omega))$. Note that $\gamma_n=(\tilde{\gamma}_n\times\id)\circ(\id\times z_n)$, which follows that $\gamma_n^*(f)(x)=(z_n)_x^*(f(\tilde{\gamma}_n^{-1}(x)))$ for every $x\in K$. Since $(z_n)^*_x({\bf 1}_{K_0(C(\Omega))})={\bf 1}_{K_0(C(\Omega))}$ for all $x\in K$, it follows that $\gamma_n^*(f)={\bf 1}_{K_0(C(\Omega))}|_{\tilde{\gamma}_n(O)}$ for $f={\bf 1}_{K_0(C(\Omega))}|_O$ and $n\geq N$. Hence, by taking $W=\tilde{\gamma}_N(O)$, we show the lemma.
\end{proof}
\begin{lem}\label{6.4}
For every $f\in C(K,\mathbb{Z})$, Define a function $F\in C(K,K_0(C(\Omega)))$
\[F(x)=f(x){\bf 1}_{K_0(C(\Omega))}\]
Then there exists $G\in C(K,K_0(C(\Omega)))$ such that $\kappa_0((j_\alpha)_{*0}(F))=(j_\beta)_{*0}(G)$.
\end{lem}
\begin{proof}
Since $f$ is continuous and $\mathbb{Z}$ is discrete, assume $f(K)=\{a_1,a_2,\cdots,a_m\}$ and let $O_i=f^{-1}(\{a_i\})$ for $1\leq i\leq m$. Then we can write $f$ into a linear combination with coefficients $a_i$
\[f=\sum_{i=1}^m a_i\chi_{O_i}\]
such that $O_i\cap O_j=\varnothing\,(i\neq j)$. This follows that $F=\sum_{i=1}^m a_i{\bf 1}_{K_0(C(\Omega))}|_{O_i}$. 
For every $i\in\{1,2,\cdots,m\}$, according to Lemma \ref{6.3}, there is an integer $N_i\geq1$ such that 
\[\kappa_0([{\bf 1}_{K_0(C(\Omega))}|_{O_i}]_\af)=(j_\beta)_{*0}({\bf 1}_{K_0(C(\Omega))}|_{\tilde{\gamma}_n(O_i)})\]
for all $n\geq N_i$. Take $N=\max\{N_1,N_2,\cdots,N_m\}$. Then
\[\kappa_0((j_\alpha)_{*0}(F))=\sum_{i=1}^m a_i\kappa_0([{\bf 1}_{K_0(C(\Omega)))}|_{O_i}])=\sum_{i=1}^m a_i (j_\beta)_{*0}({\bf 1}_{K_0(C(\Omega))}|_{\tilde{\gamma}_n(O_i)})\]
for all $n\geq N$. Also note that for all $n\geq N$, the clopen sets $\tilde{\gamma}_n(O_i)\, (i=1,2,\cdots, m)$ are disjoint whose union is $K$. By taking $G=\sum_{i=1}^m a_i{\bf 1}_{K_0(C(\Omega))}|_{\tilde{\gamma}_N(O_i)}$, we complete the proof.
\end{proof}
\begin{lem}\label{6.5}
Let $O_1, O_2\subset K$ be clopen sets. Assume that there is a function $G\in C(K,K_0(C(\Omega)))$ with 
\[{\bf 1}_{K_0(C(\Omega))}|_{O_1}-{\bf 1}_{K_0(C(\Omega))}|_{O_2}=G-G\circ\tilde{\af}^{-1}\]
Then $[\chi_{O_1}]_{\taf}=[\chi_{O_2}]_{\taf}$ in $K^0(K,\taf)$.
\end{lem}
\begin{proof}
By replacing $G$ with $G-G(x_0)$, without loss of generality, we assume $G(x_0)=0$ at some $x_0\in K$. Then there is an integer $m_1\in\mathbb{Z}$ such that 
\[G(\taf(x_0))=G(\taf(x_0))-G(x_0)=m_1{\bf 1}_{K_0(C(\Omega))}\]
Taking a sum yields that there exists a sequence of integers $\{m_n\}_{n\geq1}$ such that 
\[G(\taf^n(x_0))=m_n{\bf 1}_{K_0(C(\Omega)))}\]
Since $\taf$ is minimal, the positive orbit of $x_0$ is dense in $K$. This follows that $G=g{\bf 1}_{K_0(C(\Omega))}$
for some $g\in C(K,\mathbb{Z})$. Therefore, according to the assumption,
\[[(\chi_{O_1}-\chi_{O_2})-(g-g\circ\taf)]{\bf 1}_{K_0(C(\Omega))}=0\]
Since $K_0(C(\Omega))$ is torsion free, we have $\chi_{O_1}-\chi_{O_2}=g-g\circ\taf$ and therefore $[\chi_{O_1}]_{\taf}=[\chi_{O_2}]_{\taf}$ in $K^0(K,\taf)$.
\end{proof}

\begin{lem}\label{6.6}
Define a map $\imath_{\tilde{\alpha}}: K^0(K,\tilde{\alpha})\to K_0(A_\alpha)$ by $\imath_{\tilde{\alpha}}([f]_{\tilde{\alpha}})=[F]_{\af}$ where $F(x)=f(x){\bf 1}_{K_0(C(\Omega))}$. Then $\imath_{\tilde{\alpha}}$ is a well-defined positive injective homomorphism.
\end{lem}
\begin{proof}
We prove the lemma by the following steps.
\vspace{0.1cm}

\noindent (1) $\imath_{\tilde{\alpha}}$ is a well-defined group homomorphism: Let $f_1,f_2\in C(K,\mathbb{Z})$ such that there is $g\in C(K,\mathbb{Z})$ for which $f_1-f_2=g-g\circ\tilde{\alpha}^{-1}$. Define $G\in C(K,K_0(C(\Omega)))$ by
\[G(x)=g(x){\bf1}_{K_0(C(\Omega))}\]
Then we have $F_1(x)-F_2(x)=(f_1(x)-f_2(x)){\bf 1}_{K_0(C(\Omega))}=[g(x)-g(\tilde{\alpha}^{-1}(x))]{\bf 1}_{K_0(C(\Omega))}=G(x)-G\circ\tilde{\alpha}^{-1}(x)$, hence $[F_1]_\af=[F_2]_\af$ in $K_0(A_\alpha)$, that is, $\imath_{\tilde{\alpha}}$ is well-defined. It is straightforward to check that $\imath_{\tilde{\alpha}}$ is a group homomorphism.
\vspace{0.1cm}

\noindent (2) $\imath_{\tilde{\alpha}}$ is  positive: Let $[f]_{\taf}\in K^0(K,\tilde{\alpha})_+$, that is, we can choose $f$ such that $f(x)\in\mathbb{Z}_+\cup\{0\}$ for all $x\in K$. Then it is clear that $F(x)\in K_0(C(\Omega))_+$ for all $x\in K$ and hence $[F]_\af\geq 0$ in $K_0(A_\alpha)$.
\vspace{0.1cm}

\noindent (3) $\imath_{\taf}$ is injective: Let $p,q\in C(K,\Z)$ such that $\imath_{\taf}([p]_{\taf})=\imath_{\taf}([q]_{\taf})$. Write $p-q=\sum_{i=1}^m a_i\chi_{O_i}$ for some $a_i\in\Z$ and nonempty clopen disjoint $O_i\,(i=1,\cdots,m)$. Then 
\[\imath_{\taf}([p-q]_{\taf})=\sum_{i=1}^ma_i[{\bf 1}_{K_0(C(\Omega))}|_{O_i}]=0\ \ {\rm in}\ K_0(A_\af)\]
This follows that there is $G\in C(K,K_0(C(\Omega)))$ such that $\sum_{i=1}^m a_i{\bf 1}_{K_0(C(\Omega))}|_{O_i}=G-G\circ\taf^{-1}$. It follows similarly to Lemma \ref{6.5} that $G$ takes values in $\Z{\bf 1}_{K_0(C(\Omega))}$ and therefore the assumption that $K_0(C(\Omega))$ is torsion free implies that 
\[p-q=\sum_{i=1}^ma_i\chi_{O_i}=g-g\circ\taf^{-1}\]
for some $g\in C(K,\Z)$. Now let $P,Q\in K^0(K,\taf)$ such that $\imath_{\taf}(P)=\imath_{\taf}(Q)$. Take $p,q\in C(K,\Z)$ with $[p]_{\taf}=P$ and $[q]_{\taf}=Q$.  Then $\imath_{\taf}([p]_{\taf})=\imath_{\taf}([q]_{\taf})$. From the argument above, $[p]_{\taf}=[q]_{\taf}$ in $K^0(K,\taf)$. In other words, $P=Q$.
\end{proof}

\begin{lem}\label{6.7}
For $f\in C(K,\mathbb{Z})$, denote the element $[F]_\af\in K_0(A_\alpha)$ in Lemma \ref{6.6} by $\jmath_{\alpha}(f)$. Then $\jmath_{\alpha}: C(K,\mathbb{Z})\to K_0(A_\alpha)$ is a positive homomorphism such that 
\[\jmath_{\alpha}(C(K,\mathbb{Z})_+\setminus\{0\})\subset K_0(A_\alpha)_+\setminus\{0\}\]
Additionally, there is an order unital isomorphism
\[\vartheta: C(K,\mathbb{Z})\to C(K,\mathbb{Z})\]
such that $\kappa_0\circ \jmath_{\alpha}=\jmath_{\beta}\circ \vartheta$, i.e, there exists the following commutative diagram:
\begin{center}
\begin{tikzcd}
C(K,\mathbb{Z})   \ar[r, "\vartheta"]  \ar[d, rightarrow, "\jmath_\alpha"'] \ar[dr, phantom, ] & C(K,\mathbb{Z}) \ar[d, rightarrow, "\jmath_\beta"]\\
K_0(A_\alpha) \ar[r, rightarrow, "\kappa_0"] & K_0(A_\beta)
\end{tikzcd}
\end{center}
\end{lem}
\begin{proof}
Note that $\jmath_{\af}=\imath_{\taf}\circ (j_{\taf})_{*0}$.  It is clear that for any Cantor minimal system $(K,\taf)$, one has $(j_{\taf})_{*0}(C(K,\Z)_+)\subset K_0(A_{\taf})_+\setminus\{0\}$, and therefore it follows from Lemma \ref{6.6} that $\jmath_{\alpha}(C(K,\mathbb{Z})_+\setminus\{0\})\subset K_0(A_\alpha)_+\setminus\{0\}$. The proof of the next part is based on Lemma \ref{6.3}, Lemma \ref{6.4} and the argument in Theorem 2.6 of \cite{LM1}. However, we would like to brief the proof for readers' convenience. It goes like this.

Since $K$ is the Cantor space, one may write $\displaystyle C(K,\Z)=\lim(\Z^{k(n)},\zeta_n)$, where $k(1)\leq k(2)\leq\cdots\to\infty$ is a sequence of positive integers and $\zeta_n: \Z^{k(n)}\to \Z^{k(n+1)}$ is a sequence of unital order injective homomorphisms. For every $n\geq1$, let 
\[e_1^{(n)}=(1,0,\cdots,0), e_2^{(n)}=(0,1,0,\cdots,0), \cdots, e_{k(n)}^{(n)}=(0,0,\cdots,1)\]
be the set of generators of $\Z^{k(n)}$. Note that this could be chosen by taking a sequence of clopen partitions $\{\Pj_n\}_{n\geq1}$ such that $\#\Pj_n=k(n)$, $\Pj_{n+1}$ refines $\Pj_n$ and $\bigcup_{n}\Pj_n$ generates the topology of $K$. Then $\{e_i^{(n)}\}_{i=1}^{k(n)}$ are regarded as characteristic functions on elements in $\Pj_n$. Therefore, one regard elements in $\Z^{k(n)}$ as functions in $C(K,\Z)$ which take constant on every element of $\Pj_n$. Keep these identifications. For every $n\geq1$, let
\[x(n,i)=\kappa_0\circ \jmath_{\af}(e_i^{(n)})\]
By invoking Lemma \ref{6.3} and Lemma \ref{6.4}, there are orthogonal projections $q_i\in C(K,\Z)$ such that $x(n,i)=\jmath_{\bt}(q_i)$ and $\sum_{i=1}^{k(n)}q_i={\bf 1}_{C(K,\Z)}$. Since we have shown $\jmath_{\alpha}(C(K,\mathbb{Z})_+\setminus\{0\})\subset K_0(A_\alpha)_+\setminus\{0\}$ and $\kappa_0$ is a unital order isomorphism, $x(n,i)\neq0$ and hence $q_i$ are clearly nonzero. This gives a positive unital homomorphism $\vartheta_n: \Z^{k(n)}\to C(K,\Z)$ such that 
\[\kappa_0\circ \jmath_{\af}|_{\Z^{k(n)}}=\jmath_{\bt}\circ \vartheta_n\]
Note that $\vartheta_n$ are injective because $q_i$ are nonzero and orthogonal. One can apply a similar argument by writing 
\[e_i^{(n)}=\sum_{j\in S(n,i)}e_j^{(n+1)}\]
for some disjoint sets $S(n,1),\cdots, S(n,k)$ such that $\bigcup_{i}S(n,i)=\{1,2,\cdots,k(n+1)\}$ and letting $p_i=\vartheta_n(e_i^{(n)})$. This yields that there exists orthogonal nonzero projections $p(n+1,j)\in C(K,\Z)_+\,(j\in S(n,i))$ such that
\[\sum_{j\in S(n,i)}p(n+1,j)=p_i=\vartheta_n(e_i^{(n)})\ \ {\rm and}\ \ \jmath_{\beta}(p(n+1,j))=\kappa_0\circ\jmath_{\alpha}(e_j^{(n+1)})\]
This gives us a sequence of order unital injective homomorphisms $\{\vartheta_n\}$ such that  $\kappa_0\circ\jmath_{\alpha}|_{\mathbb{Z}^{k(n)}}=\jmath_{\beta}\circ \vartheta_n$ for all $n\geq1$, which induces an order unital injective homomorphism $\vartheta: C(K,\Z)\to C(K,\Z)$ such that $\kappa_0\circ \jmath_{\af}=\jmath_{\bt}\circ\vartheta$. Finally, by applying a similar argument to $\kappa_0^{-1}\circ \jmath_{\bt}$ shows that $\vartheta$ has an order unital inverse. For more details, one may see Theorem 2.6 in \cite{LM1}.
\end{proof}
\begin{lem}\label{6.8}
There exists a unital positive isomorphism $\Theta: K^0(K,\tilde{\af})\to K^0(K,\tilde{\bt})$ such that the following diagram commutes.

$$\xymatrix@C=1.8cm{C(K,\mathbb{Z})\ar[r]^{\vartheta}\ar[d]_{(j_{\tilde{\af}})_{*0}} \ar@/_45pt/[dd]_-{\jmath_{\alpha}} & C(K,\mathbb{Z})\ar[d]^{(j_{\tilde{\bt}})_{*0}}\ar@/^45pt/[dd]^-{\jmath_{\bt}} \\
K^0(K,\tilde{\af}) \ar[r]^{\Theta} \ar[d]_{\imath_{\tilde{\af}}} & K^0(K,\tilde{\bt})\ar[d]^{\imath_{\tilde{\bt}}}\\
K_0(A_\af) \ar[r]^{\kappa_0} & K_0(A_\bt)}$$
where $\kappa_0$ is the unital positive isomorphism in Lemma \ref{6.3}, $\vartheta$ is the unital positive isomorphism defined in Lemma \ref{6.7}, $\imath_{\tilde{\af}}$ and $\imath_{\tilde{\bt}}$ are unital positive homomorphisms defined in Lemma \ref{6.6}.
\end{lem}
\begin{proof}
We first define the map $\Theta: K^0(K,\tilde{\af})\to K^0(K,\tilde{\bt})$. Let $P\in K^0(K,\tilde{\alpha})$. Then there exists $p\in C(K,\mathbb{Z})$ with $(j_{\tilde{\af}})_{*0}(p)=P$. Define
\[\Theta(P)=(j_{\tilde{\bt}})_{*0}\circ\vartheta(p)\]
We now verify that $\Theta$ is a well-defined unital positive isomorphism. Let $p_1,p_2\in C(K,\mathbb{Z})$ with $P=(j_{\tilde{\af}})_{*0}(p_1)=(j_{\tilde{\af}})_{*0}(p_2)$. We shall show that $(j_{\tilde{\bt}})_{*0}\circ\vartheta(p_1)=(j_{\tilde{\bt}})_{*0}\circ\vartheta(p_2)$. Note that $\jmath_\alpha=\imath_{\tilde{\af}}\circ(j_{\tilde{\af}})_{*0}$, which follows that $\jmath_{\af}(p_1)=\jmath_{\af}(p_2)$ 
and hence $\kappa_0\circ\jmath_{\af}(p_1)=\kappa_0\circ\jmath_{\af}(p_2)$. According to the commutative diagram in Lemma \ref{6.7}, we have $\jmath_{\bt}\circ\vartheta(p_1)=\jmath_{\bt}\circ\vartheta(p_2)$, that is,
\[\imath_{\tilde{\bt}}\circ(j_{\tilde{\bt}})_{*0}\circ\vartheta(p_1)=\imath_{\tilde{\bt}}\circ(j_{\tilde{\bt}})_{*0}\circ\vartheta(p_2)\]
According to Lemma \ref{6.6}, $\imath_{\tilde{\bt}}$ is injective, which follows that $(j_{\tilde{\bt}})_{*0}\circ\vartheta(p_1)=(j_{\tilde{\bt}})_{*0}\circ\vartheta(p_2)$. This proves that $\Theta$ is well-defined.

Now let $P_1, P_2\in K^0(K,\tilde{\af})$. Choose $p_1, p_2\in C(K,\mathbb{Z})$ such that $(j_{\tilde{\af}})_{*0}(p_i)=P_i\,(i=1,2)$. According to the assumption, we then have
\[\Theta(P_1)+\Theta(P_2)=(j_{\tilde{\bt}})_{*0}\circ\vartheta(p_1)+(j_{\tilde{\bt}})_{*0}\circ\vartheta(p_2)=(j_{\tilde{\bt}})_{*0}\circ\vartheta(p_1+p_2)\]
Since $(j_{\tilde{\af}})_{*0}(p_1+p_2)=P_1+P_2$, this is equal to $\Theta(P_1+P_2)$, and hence $\Theta$ is a group homomorphism. It follows directly from the construction that $\Theta$ is unital. Since $(j_{\tilde{\af}})_{*0}(C(K,\mathbb{Z}_+))=K^0(K,\tilde{\af})_+$, $\Theta$ is positive. It remains to show that $\Theta$ is a bijection. Let $Q\in K^0(K,\tilde{\beta})$ and $q\in C(K,\mathbb{Z})$ such that $(j_{\tilde{\bt}})_{*0}(q)=Q$. Then it is easy to check that $(j_{\tilde{\af}})_{*0}\circ \vartheta^{-1}(q)\in K^0(K,\tilde{\af})$ and 
\[\Theta((j_{\tilde{\af}})_{*0}\circ \vartheta^{-1}(q))=Q\]
Hence $\Theta$ is surjective. Finally, take $P_1, P_2\in K^0(K,\tilde{\af})$ such that $\Theta(P_1)=\Theta(P_2)$.
Note that in order to show that $P_1=P_2$, it suffices to show the commutativity of the lower cell, for once this is done, it will follow that $\kappa_0\circ\imath_{\tilde{\af}}(P_1)=\kappa_0\circ\imath_{\tilde{\af}}(P_2)$ and therefore according to Lemma \ref{6.6}, $P_1=P_2$. Now we verify that
\[\kappa_0\circ\imath_{\tilde{\af}}=\imath_{\tilde{\bt}}\circ\Theta\]
Let $P\in K^0(K,\tilde{\af})$. Then there is $p\in C(K,\mathbb{Z})$ such that $(j_{\tilde{\af}})_{*0}(p)=P$ and 
\[\Theta(P)=(j_{\tilde{\bt}})_{*0}\circ\vartheta(p)\]
This follows that $\imath_{\tilde{\bt}}\circ\Theta(P)=\imath_{\tilde{\bt}}\circ(j_{\tilde{\bt}})_{*0}\circ\vartheta(p)=\jmath_{\bt}\circ\vartheta(p)$. According to Lemma \ref{6.7},
$\jmath_{\bt}\circ\vartheta(p)=\kappa_0\circ \jmath_{\alpha}(p)=\kappa_0\circ\imath_{\tilde{\af}}\circ(j_{\tilde{\af}})_{*0}(p)=\kappa_0\circ\imath_{\tilde{\af}}(P)$, 
which implies that $\kappa_0\circ\imath_{\tilde{\af}}=\imath_{\tilde{\bt}}\circ\Theta$ and the proof is complete.
\end{proof}
\begin{cor}\label{6.9}
Let $(K,\taf)$ and $(K,\tbt)$ be Cantor minimal systems, $\af=\taf\times c_1$ and $\bt=\tbt\times c_2$ be minimal homeomorphisms on $X=K\times\Omega$ with quasi-trivial cocycles. Then the $C^*$-strongly approximate conjugacy of $(X,\af)$ and $(X,\bt)$ implies the approximate $K$-conjugacy of $(K,\taf)$ and $(K,\tbt)$.
\end{cor}
\begin{proof}
According to Lemma \ref{6.8}, there is a unital order isomorphism
\[\Theta: (K^0(K,\taf),K_0(K,\taf)_+,[1_K])\to (K^0(K,\tbt),K_0(K,\tbt)_+,[1_K])\]
Then $(K,\taf)$ and $(K,\tbt)$ are approximately $K$-conjugate by Theorem \ref{5.13}, Remark \ref{5.14} and Theorem 5.4 in \cite{LM1}.
\end{proof}
We now begin to deal with the approximate $K$-conjugacy of $(X,\af)$ and $(X,\bt)$. First, we have the following lemma (also note that $K_1(C(\Omega))$ is not necessarily trivial in this lemma).
\begin{lem}\label{6.10}
Let $(K,\tilde{\alpha})$ and $(K,\tilde{\beta})$ be Cantor minimal systems. Let $\alpha=\tilde{\alpha}\times c_1\in\Raf$ and $\bt=\tilde{\bt}\times c_2\in\Rbt$ be two minimal homeomorphisms of $X=K\times\Omega$. Then the followings are equivalent.

\noindent $(1)$ $(X, \af)$ and $(X, \bt)$ are weakly approximately conjugate;

\noindent $(2)$ $(K,\tilde{\alpha})$ and $(K,\tilde{\bt})$ are weakly approximately conjugate;

\noindent $(3)$ $(X, \af)$ and $(X, \bt)$ are weakly approximately conjugate via $\sigma'_n,\gamma'_n$ whose cocycles are quasi-trivial.
\end{lem}
\begin{proof}
$(1)\Rightarrow (2)$: This is a straightforward verification. Let $(X, \af)$ and $(X, \bt)$ be weakly approximately conjugate via $\sigma_n,\gamma_n$. Write $\sigma_n=\tilde{\sigma}_n\times y_n$ and $\gamma_n=\tilde{\gamma}_n\times z_n$ respectively. To show that $(K,\tilde{\alpha})$ and $(K,\tilde{\bt})$ are weakly approximately conjugate, it suffices to show that for every $\varepsilon>0$, there are homeomorphisms $\nu:K\to K$ and $\nu':K\to K$ such that for all $x\in K$,
\[d_K(\nu\circ\tilde{\alpha}\circ\nu^{-1}(x),\tilde{\bt}(x))<\varepsilon\ \ {\rm and}\ \ d_K(\nu'\circ\tilde{\bt}\circ\nu'^{-1}(x),\tilde{\af}(x))<\varepsilon\]
Without loss of generality, we show the former. Since $(K\times\Omega, \tilde{\alpha}\times c_1)$ and $(K\times\Omega, \tilde{\bt}\times c_2)$ are weakly approximately conjugate, there is $N\geq1$ such that
\[d_{K\times\Omega}(\sigma_N\circ(\tilde{\af}\times c_1)\circ\sigma_N^{-1}(x,\omega), \tilde{\bt}\times c_2(x,\omega))<\varepsilon\]
for all $(x,\omega)\in K\times\Omega$. However, note that
\begin{align*}
&d_{K\times\Omega}(\sigma_N\circ(\tilde{\af}\times c_1)\circ\sigma_N^{-1}(x,\omega), \tilde{\bt}\times c_2(x,\omega))\\
=\,&d_{K\times\Omega}\left((\tilde{\sigma}_N\circ\tilde{\alpha}\circ\tilde{\sigma}_N^{-1}(x), (y_N)_{\tilde{\af}\circ\tilde{\sigma}_N^{-1}(x)}\circ(c_1)_{\tilde{\sigma}_N^{-1}(x)}\circ(y_N)_{\tilde{\sigma}_N^{-1}(x)}^{-1}(\omega)), (\tilde{\bt}(x), (c_2)_{x}(\omega))\right)\\
\geq\,&d_{K}(\tilde{\sigma}_N\circ\tilde{\alpha}\circ\tilde{\sigma}_N^{-1}(x), \tilde{\bt}(x))
\end{align*}
which follows $(2)$.
\vspace{0.1cm}

\noindent $(2)\Rightarrow (3)$: Assume that  $(K,\tilde{\af})$ and $(K,\tilde{\bt})$ are weakly approximately conjugate via $\tilde{\sigma}_n,\tilde{\gamma}_n$. Then it is clear that $(K\times\Omega,\tilde{\af}\times{\rm id})$ and $(K\times\Omega, \tilde{\bt}\times{\rm id)}$ are weakly approximately conjugate via $\tilde{\sigma}_n\times{\rm id},\tilde{\gamma}_n\times{\rm id}$. On the other hand, according to Lemma \ref{5.3}, $(K\times\Omega,\tilde{\af}\times{\rm id})$ is weakly approximately conjugate to $(K\times\Omega,\tilde{\af}\times c_1)$ via ${\rm id}\times g_n, {\rm id}\times h_n$, and $(K\times\Omega,\tilde{\bt}\times{\rm id})$ is weakly approximately conjugate to $(K\times\Omega,\tilde{\bt}\times c_2)$ via ${\rm id}\times g_n', {\rm id}\times h_n'$. Then by Proposition \ref{5.4}, we see that $(K\times\Omega, \tilde{\alpha}\times c_1)$ and $(K\times\Omega, \tilde{\bt}\times c_2)$ are weakly approximately conjugate via subsets
\[S_1\subset\{({\rm id}\times g'_l)\circ(\tilde{\sigma}_m\circ{\rm id})\circ({\rm id}\times h_n): l,m,n\geq1 \}\]
and
\[S_2\subset\{({\rm id}\times g_l)\circ(\tilde{\gamma}_m\circ{\rm id})\circ({\rm id}\times h'_n): l,m,n\geq1 \}\]
Since cocycles $g_n,g'_n,h_n,h'_n$ all take values in $G_\mathcal{L}$, they are quasi-trivial. 
\vspace{0.1cm}

\noindent $(3)\Rightarrow (1)$: This is immediate.
\end{proof}
Recall that for a dynamical system $(Z,\af)$, its {\it periodic spectrum} is defined by
\[\textstyle PS(\af)=\{p\in\mathbb{Z}_+: \exists\ {\rm clopen\ set}\ U\ {\rm with}\ \bigsqcup_{i=0}^{p-1}\af^i(U)=Z\}\]
\begin{cor}\label{6.11}
Keep the assumptions as in Lemma \ref{6.10}. Then $\af$ and $\bt$ are weakly approximately conjugate if and only if $PS(\af)=PS(\bt)$.
\end{cor}
\begin{proof}
According to Lemma \ref{6.10}, $(X,\af)$ and $(X,\bt)$ are weakly approximately conjugate if and only if $(K,\tilde{\af})$ and $(K,\tilde{\bt})$ are weakly approximately conjugate. Since $\Omega$ is connected, every clopen set $U\subset X$ is of the form $O\times \Omega$ for some clopen set $O\subset K$, which follows that $PS(\af)=PS(\tilde{\af})$ and $PS(\bt)=PS(\tilde{\bt})$. Finally, according to Theorem 4.13 in \cite{LM1}, $PS(\tilde{\af})=PS(\tilde{\bt})$ if and only if $\tilde{\af}$ and $\tilde{\bt}$ are weakly approximately conjugate. This proves the corollary.
\end{proof}

The following proposition follows directly from Theorem \ref{5.13},
\begin{prop}\label{6.12}
Let $\af$ and $\bt$ be rigid minimal homeomorphisms of $X$. Then $(X,\af)$ and $(X,\bt)$ are approximately $K$-conjugate if and only if there exists a unital order isomorphism 
\[\kappa: (K_0(A_\af), K_0(A_\af)_+, [1_{A_\af}], K_1(A_\af))\to (K_0(A_\bt), K_0(A_\bt)_+, [1_{A_\bt}], K_1(A_\bt))\]
and two sequences of homeomorphisms $\chi_n,\chi'_n: X\to X$ such that $(X,\af)\sim_{\rm app}(X,\bt)$ via $\chi_n,\chi'_n$, and for every finitely-generated subgroups $G_i, F_i\subset K_i(C(X))$,
\[\kappa_i\circ(j_\af)_{*i}|_{G_i}=(j_{\bt}\circ\chi^{-1}_n)_{*i}|_ {G_i}\ \ {\rm and}\ \ \kappa_i^{-1}\circ(j_\bt)_{*i}|_{F_i}=(j_{\af}\circ\chi'^{-1}_n)_{*i}|_ {F_i}\]
for $i=0,1$ and all sufficiently large $n$.
\end{prop}
\begin{lem}\label{6.13}
Let $(K,\tilde{\alpha})$ and $(K,\tilde{\beta})$ be Cantor minimal systems. Let $\alpha=\tilde{\alpha}\times c_1\in\Raf$ and $\bt=\tilde{\bt}\times c_2\in\Rbt$ be two minimal rigid homeomorphisms of $X=K\times\Omega$. Then the approximate $K$-conjugacy of $(K,\taf)$ and $(K,\tbt)$ implies the approximate $K$-conjugacy of $(X,\af)$ and $(X,\bt)$.
\end{lem}
\begin{proof}
Suppose $(K,\tilde{\af})$ and $(K,\tilde{\bt})$ are approximately $K$-conjugate via $\tilde{\sigma}_n$ and $\tilde{\gamma}_n$. Hence they are weakly approximately conjugate via $\tilde{\sigma}_n$ and $\tilde{\gamma}_n$. According to Lemma \ref{6.10}, $(K\times\Omega,\tilde{\af}\times c_1)$ and $(K\times\Omega,\tilde{\bt}\times c_2)$ are weakly approximately conjugate via $\sigma_n,\gamma_n$, which has the form
\[\sigma_n=\tilde{\sigma}_n\times g_n\ \ {\rm and}\ \ \gamma_n=\tilde{\gamma}_n\times h_n\]
with the cocycles $g_n$ and $h_n$ taking values in $G_\mathcal{L}$. We assert that  $(K\times\Omega,\tilde{\af}\times c_1)$ and $(K\times\Omega,\tilde{\bt}\times c_2)$ are actually approximately $K$-conjugate via $\sigma_n$ and $\gamma_n$. 

Since $(K,\tilde{\af})$ and $(K,\tilde{\bt})$ are approximately $K$-conjugate via $\tilde{\sigma}_n$ and $\tilde{\gamma}_n$, it follows from Theorem 5.4 of \cite{LM1} that there is a unital order isomorphism
\[\Theta: (K^0(K,\tilde{\af}), K^0(K,\tilde{\af})_+, [1_{A_{\tilde{\af}}}])\to (K^0(K,\tilde{\bt}), K^0(K,\tilde{\bt})_+, [1_{A_{\tilde{\bt}}}])\]
such that for every $p\in C(K,\mathbb{Z})$ there is $N\geq1$ with
\[\Theta\circ(j_{\tilde{\af}})_{*0}(p)=(j_{\tilde{\bt}})_{*0}(p\circ \tilde{\gamma}_n^{-1}),\ \ {\rm and}\ \ \Theta^{-1}\circ(j_{\tilde{\bt}})_{*0}(p)=(j_{\tilde{\af}})_{*0}(p\circ \tilde{\sigma}_n^{-1})\]
for all $n\geq N$. For $i=0,1$, we define $\kappa_i: K_i(A_{\af})\to K_i(A_{\bt})$ by $\kappa_0=\Theta\otimes {\rm id}_{K_0(C(\Omega))}$ and $\kappa_1={\rm id}_{K_0(C(\Omega))}$.
Then it is clear that 
\[\kappa=(\kappa_0,\kappa_1): (K_0(A_\af), K_0(A_\af)_+, [1_{A_\af}], K_1(A_\af))\to (K_0(A_\bt), K_0(A_\bt)_+, [1_{A_\bt}], K_1(A_\af))\]
defines an isomorphism. We now assert that for every finitely-generated subgroups $G_i, F_i\subset K_i(C(X))$,
$\kappa_i\circ(j_\af)_{*i}|_{G_i}=(j_{\bt}\circ\gamma^{-1}_n)_{*i}|_ {G_i}$ and $\kappa_i^{-1}\circ(j_\bt)_{*i}|_{F_i}=(j_{\af}\circ\sigma^{-1}_n)_{*i}|_ {F_i}$.
for $i=0,1$ and all sufficiently large $n$. It suffices to show that, for every $x\in K_i(C(X))$, there is $N\geq1$ such that
\[\kappa_i\circ(j_\af)_{*i}(x)=(j_\bt\circ\gamma_n^{-1})_{*i}(x)\ \ {\rm and}\ \ \kappa_{i}^{-1}\circ(j_\bt)_{*i}(x)=(j_\af\circ\gamma_n^{-1})_{*i}(x)\]
for $i=0,1$ and all $n\geq N$. Since $K_1(C(X))=0$, $(j_\af)_{*1}=(j_\bt)_{*1}=0$, and we only need to show the case of $i=0$. Without loss of generality, we show the former.
\vspace{0.1cm}

Let $x\in K_0(C(X))$. Write $x=\sum_{s\in S}F_s\otimes p_s$, 
where $S$ is a finite set, $F_s\in C(K,\mathbb{Z})$ and $p_s\in K_0(C(\Omega))$ for all $s\in S$. Then we have
\[(j_\af)_{*0}(x)=\left[\sum_{s\in S}F_s\otimes p_s\right]_\af=\sum_{s\in S}[F_s]_{\tilde{\af}}\otimes p_s\]
 According to the definition of $\kappa_0$, we have $\kappa_0\circ(j_\af)_{*0}(x)=\sum_{s\in S}\Theta([F_s]_{\tilde{\af}})\otimes p_s$.
On the other hand, since $\gamma_n=\tilde{\gamma}_n\times h_n$ and $h_n$ are quasi-trivial, we have
\[(j_\bt\circ\gamma_n^{-1})_{*0}(x)=\sum_{s\in S}([F_s\circ\tilde{\gamma}_n^{-1}]_{\tilde{\bt}})\otimes p_s\]
According to the assumption, for every $s\in S$, there is $N_s\geq1$ such that $\Theta([F_s]_{\tilde{\af}})=[F_s\circ\tilde{\gamma}_n^{-1}]_{\tilde{\bt}}$.
Since $S$ is a finite set, let $N=\max\{N_s: s\in S\}$. This follows that 
$\kappa_0\circ(j_\af)_{*0}(x)=(j_\bt\circ\gamma_n^{-1})_{*0}(x)$ for all $n\geq N$.  Finally, according to Proposition \ref{6.12}, $(K\times\Omega,\tilde{\af}\times c_1)$ and $(K,\times\Omega,\tilde{\bt}\times c_2)$ are approximately $K$-conjugate via $\sigma_n,\gamma_n$, whose cocycles take values in $G_\mathcal{L}$ (which are quasi-trivial). This completes the proof.
\end{proof}
\begin{lem}\label{6.14}
Let $(K,\tilde{\alpha})$ and $(K,\tilde{\beta})$ be Cantor minimal systems. Let $\alpha=\tilde{\alpha}\times c_1\in\Raf$ and $\bt=\tilde{\bt}\times c_2\in\Rbt$ be two minimal rigid homeomorphisms of $X=K\times\Omega$. If there is a unital order isomorphism
\[\kappa: (K_0(A_\alpha), K_0(A_\alpha)_+, [1_{A_\alpha}], K_1(A_\alpha))\to(K_0(A_\bt), K_0(A_\bt)_+, [1_{A_\bt}], K_1(A_\bt))\]
and a homeomorphism $\chi: X\to X$ such that $\kappa_0\circ (j_\af)_{*0}=(j_\bt\circ\chi^{-1})_{*0}$, then there exist unital order isomorphisms $\vartheta': C(K,\mathbb{Z})\to C(K,\mathbb{Z})$ and $\Theta': K^0(K,\tilde{\af})\to K^0(K,\tilde{\bt})$, 
unital order injective homomorphisms
\begin{align*}
\cdot{\bf 1}_{K_0(C(\Omega))}: C(K,\mathbb{Z})&\to K_0(C(X)),\\
\imath'_{\tilde{\af}}: K^0(K,\tilde{\af})&\to K_0(A_\af)\ {\rm and},\\
\imath'_{\tilde{\bt}}: K^0(K,\tilde{\bt})&\to K_0(A_\bt)
\end{align*}
such that every cell of the following diagram is commutative. 
$$\xymatrix@C=0.5cm@R=1.3cm{ & C(K,\mathbb{Z})\ar[rr]^{\vartheta'}\ar[ld]_{\cdot{\bf 1}_{K_0(C(\Omega))}}\ar@{-}[d]^<<<<<<<<{[\cdot]_{\tilde{\af}}} & & C(K,\mathbb{Z})\ar[dd]^<<<<<<<<<{[\cdot]_{\tilde{\bt}}}\ar[ld]_{\cdot{\bf 1}_{K_0(C(\Omega))}} \\
K_0(C(X)) \ar[rr]^{\ \ \ \ \ \ \ (\chi^*)_{*0}}\ar[dd]_<<<<<<<<<<<<<{(j_\af)_{*0}} & \ar[d] & K_0(C(X)) \ar[dd]^<<<<<<<<<<<<<{(j_\bt)_{*0}} \\
& K^0(K,\tilde{\af}) \ar[rr]^<<<<<<<<<{\Theta'} \ar[ld]_{\imath'_{\tilde{\af}}} && K^0(K,\tilde{\bt})\ar[ld]_{\imath'_{\tilde{\bt}}} \\
K_0(A_\af)\ar[rr]^{\ \ \kappa_0} & & K_0(A_\bt)}$$
\end{lem}
\begin{proof}
Let $f\in C(K,\mathbb{Z})$. Define $\cdot{\bf 1}_{K_0(C(\Omega))}(f)\in K_0(C(X))$ by
\[\cdot{\bf 1}_{K_0(C(\Omega))}(f)(x)=f(x){\bf 1}_{K_0(C(\Omega))}\]
It is clear that $\cdot{\bf 1}_{K_0(C(\Omega))}$ is a unital order homomorphism. 
We now define a unital order isomorphism $\vartheta': C(K,\mathbb{Z})\to C(K,\mathbb{Z})$ such that
\[(\chi^*)_{*0}\circ \cdot{\bf 1}_{K_0(C(\Omega))}(f)=\cdot{\bf 1}_{K_0(C(\Omega))}\circ \vartheta'(f)\]
 for all $f\in C(K,\mathbb{Z})$. Note that for every $f\in C(K,\mathbb{Z})$, there exists $a_1,a_2,\cdots,a_m\in\mathbb{Z}$ and clopen subsets $O_1,O_2,\cdots,O_m\subset K$ for some $m\geq1$ such that $f=\sum_{1\leq i\leq m}a_i\cdot1_O$ where $1_O$ denotes the characteristic function on $O$. Therefore it suffices to define $\vartheta'(1_O)$ such that
\[(\chi^*)_{*0}\circ \cdot{\bf 1}_{K_0(C(\Omega))}(1_O)=\cdot{\bf 1}_{K_0(C(\Omega))}\circ \vartheta'(1_O)\]
for every clopen set $O\subset K$. Write $\cdot{\bf 1}_{K_0(C(\Omega))}(1_O)={\bf 1}_{K_0(C(\Omega))}|_O$ and $\chi=\tilde{\chi}\times e$ where $\tilde{\chi}$ is a homeomorphism of $K$ and $e: K\to \mathcal{H}(\Omega)$ is the continuous cocycle of $\chi$. Note that $(\chi^*)_{*0}({\bf 1}_{K_0(C(\Omega))}|_O)=(\chi^*)_{*0}[1_{O\times\Omega}]=[\chi^*(1_{O\times\Omega})]=[1_{\tilde{\chi}(O)\times\Omega}]={\bf 1}_{K_0(C(\Omega))}|_{\tilde{\chi}(O)}$,
hence we define 
\[\vartheta'(1_O)=1_{\tilde{\chi}(O)}.\]
This gives the commutative diagram of the upper cell. It is clear that $\vartheta'$ defines a unital order isomorphism of $C(K,\mathbb{Z})$ onto itself. Now we define homomorphisms 
\[\imath'_{\tilde{\af}}: K^0(K,\tilde{\af})\to K_0(A_\af)\ \ {\rm and}\ \ \imath'_{\tilde{\bt}}: K^0(K,\tilde{\bt})\to K_0(A_\bt)\]
Let $[f]_{\tilde{\af}}\in K^0(K,\tilde{\af})$ and define
\[\imath'_{\tilde{\af}}([f]_{\tilde{\af}})=(j_\af)_{*0}\circ\cdot{\bf 1}_{K_0(C(\Omega))}(f)\]
Assume $f,g\in C(K,\mathbb{Z})$ such that $f-g=h-h\circ\tilde{\af}^{-1}$ for some $h\in C(K,\mathbb{Z})$. This follows that 
\[\cdot{\bf 1}_{K_0(C(\Omega))}(f)(x)-\cdot{\bf 1}_{K_0(C(\Omega))}(g)(x)=h(x){\bf 1}_{K_0(C(\Omega))}-h(\tilde{\af}^{-1}(x)){\bf 1}_{K_0(C(\Omega))}\]
that is, $\cdot{\bf 1}_{K_0(C(\Omega))}(f)-\cdot{\bf 1}_{K_0(C(\Omega))}(g)=\cdot{\bf 1}_{K_0(C(\Omega))}(h)-\cdot{\bf 1}_{K_0(C(\Omega))}(h)\circ\tilde{\af}^{-1}$, which shows that $\imath'_{\tilde{\af}}$ is well-defined. By applying the same argument to Lemma \ref{6.6}, one checks that $\imath'_{\tilde{\af}}$ is unital order and injective, and this gives the commutativity of the left cell. We define $\imath'_{\tilde{\bt}}$ in a similar manner to $\imath'_{\tilde{\af}}$. Now it remains to define a unital order isomorphism 
\[\Theta': K^0(K,\tilde{\af})\to K^0(K,\tilde{\bt})\]
such that $\imath'_{\tilde{\bt}}\circ\Theta'=\kappa_0\circ\imath'_{\tilde{\af}}$ and $\Theta'\circ[\cdot]_{\tilde{\af}}=[\cdot]_{\tilde{\bt}}\circ \vartheta'$. 
For $[f]_{\tilde{\af}}\in K^0(K,\tilde{\af})$, we now first show that 
\[\kappa_0\circ\imath'_{\tilde{\af}}([f]_{\tilde{\af}})\in \imath'_{\tilde{\bt}}(K^0(K,\tilde{\bt}))\]
Note that according to the commutativity of the other four cells, one has 
\begin{align*}
&\kappa_0\circ\imath'_{\tilde{\af}}([f]_{\tilde{\af}})\\
=\,&\kappa_0\circ(j_\af)_{*0}\circ\cdot{\bf 1}_{K_0(C(\Omega))}(f)\\
=\,&(j_\bt)_{*0}\circ(\chi^*)_{*0}\circ\cdot{\bf 1}_{K_0(C(\Omega))}(f)\\
=\,&(j_\bt)_{*0}\circ\cdot{\bf 1}_{K_0(C(\Omega))}\circ\vartheta'(f)\\
=\,&\imath'_{\tilde{\bt}}\circ[\vartheta'(f)]_{\tilde{\bt}}
\end{align*}
which shows that $\kappa_0\circ\imath'_{\tilde{\af}}([f]_{\tilde{\af}})\in \imath'_{\tilde{\bt}}(K^0(K,\tilde{\bt}))$. Since $\imath'_{\tilde{\bt}}$ is injective, we define
\[\Theta'([f]_{\tilde{\af}})={\imath'_{\tilde{\bt}}}^{-1}\circ\kappa_0\circ\imath'_{\tilde{\af}}([f]_{\tilde{\af}})=[\vartheta'(f)]_{\tilde{\bt}}\]
Also note that the above equation implies that $\Theta'$ is well-defined which makes the other two cells commutative. Now it suffices to show $\Theta'$ is a unital order isomorphism. Since $\vartheta'$ and $[\cdot]_{\tilde{\bt}}$ is unital and order, so is $\Theta'$. To see that $\Theta'$ is an isomorphism, note that from the definition of $\Theta'([f]_{\tilde{\af}})$,  we have
\[\Theta'([f]_{\tilde{\af}})={\imath'_{\tilde{\bt}}}^{-1}\circ\kappa_0\circ\imath'_{\tilde{\af}}([f]_{\tilde{\af}})\]
Since $\imath'_{\tilde{\af}}, \imath'_{\tilde{\bt}}$ and $\kappa_0$ are injective, so is $\Theta'$. Now let $[g]_{\tilde{\bt}}\in K^0(K,\tilde{\bt})$ be an arbitrary element. A similar argument yields that $\kappa_0^{-1}\circ\imath'_{\tilde{\bt}}([g]_{\tilde{\bt}})\in \imath'_{\tilde{\af}}(K^0(K,\tilde{\af}))$.
Let 
\[[f]_{\tilde{\af}}={\imath'_{\tilde{\af}}}^{-1}\circ\kappa_0^{-1}\circ\imath'_{\tilde{\bt}}([g]_{\tilde{\bt}})\]
Then $\Theta'([f]_{\tilde{\af}})=[g]_{\tilde{\bt}}$, 
therefore $\Theta'$ is surjective. This completes the proof.
\end{proof}
{\bf Proof of Theorem \ref{6.1}.} 
\begin{proof}
$(1)\Rightarrow (2)$: This follows directly from Theorem \ref{5.13};

$(2)\Rightarrow (1)$: According to Corollary \ref{6.9}, the $C^*$-strongly-approximate conjugacy of $(X,\af)$ and $(X,\bt)$ implies the approximate $K$-conjugacy of $(K,\taf)$ and $(K,\tbt)$, which follows the approximate $K$-conjugacy of $(X,\af)$ and $(X,\bt)$ by Lemma \ref{6.13};

$(1)\Rightarrow (3)$: Note that according to the implication $(1)\Rightarrow(2)$ and the the proof of Lemma \ref{6.13}, we can assume that $(X,\af)$ and $(X,\bt)$ are approximate $K$-conjugacy via conjugate maps whose cocyles are quasi-trivial. This follows that the unital order isomorphism 
\[\kappa:  (K_0(A_\alpha), K_0(A_\alpha)_+, [1_{A_\alpha}], K_1(A_\alpha))\to(K_0(A_\bt), K_0(A_\bt)_+, [1_{A_\bt}], K_1(A_\bt))\]
has the form $\kappa=\Theta\otimes\id$, where $\Theta: K^0(K,\taf)\to K^0(K,\tbt)$ is the map defined in Lemma \ref{6.8}. On the other hand, by Lemma \ref{6.7},  there is an order unital isomorphism $\vartheta: C(K,\Z)\to C(K,\Z)$ such that $\kappa_0\circ\jmath_{\af}=\jmath_{\bt}\circ\vartheta$. Since $C(K)$ is a unital AF-algebra, there is a homeomorphism $\tilde{\chi}: K\to K$ which induces $\vartheta$. Define $\chi=\tilde{\chi}\times\id$. This gives $(3)$;

$(3)\Rightarrow (1)$: By Lemma \ref{6.14} and Theorem 5.4 in \cite{LM1}, $(K,\taf)$ and $(K,\tbt)$ are approximately $K$-conjugate. By Lemma \ref{6.13}, $(X,\af)$ and $(X,\bt)$ are approximate $K$-conjugate.
\end{proof}

\section{The case of nontrivial $K_1$-group}\label{sec:7}
Let $K$ be the Cantor space and $\Omega$ be a compact connected finite CW-complex which satisfies the LMP. Let $\af: K\times\Omega\to K\times\Omega$ be a minimal homeomorphism in $\Raf$ and $A_{\{x_0\}}\subset A_\af$ be an orbit-breaking subalgebra as defined in Sect.~\ref{sec:3}. In this section, we drop the condition that $K_1(C(\Omega))=0$, and keep the assumption that $K_i(C(\Omega))\,(i=0,1)$ are torsion free (note that when $\Omega$ is a finite CW-complex, then $K_i(C(\Omega))\,(i=0,1)$ is always finitely-generated). Recall that according to Proposition \ref{3.3}, we have, in this situation, 
\[K_i(A_{\{x_0\}})\cong K^0(K,\tilde{\af})\otimes K_i(C(\Omega))\ \ {\rm and}\ \ K_i(A_\af)\cong K_i(A_{\{x_0\}})\oplus K_{i+1}(C(\Omega))\]
Besides, the natural embedding
\[j_1: C(K\times\Omega)\hookrightarrow A_{\{x_0\}}\]
induces a surjective unital order homomorphism of $K_0$-groups:
\[(j_1)_{*0}: C(K,\mathbb{Z})\otimes K_0(C(\Omega))\cong K_0(C(K\times\Omega))\to K_0(A_{\{x_0\}})\cong K^0(K,\tilde{\af})\otimes K_0(C(\Omega))\]
Through this section, denote $X=K\times\Omega$. Note that according to the proof of Theorem \ref{3.7}, we have the following Lemma.
\begin{lem}\label{7.1}
For any finite subset $\mathcal{F}\subset A_\af$ and $\varepsilon>0$, there exists a projection $p\in A_{\{x_0\}}$ such that 

\noindent (1) For every $a\in \mathcal{F}$, $ap\in_{\varepsilon} A_{\{x_0\}}$;

\noindent (2) $\tau(1-p)<\varepsilon$ for every tracial state $\tau\in T(A_{\{x_0\}})$.
\end{lem}
\begin{lem}\label{7.2}
Assume $x_0\in K$ and ${\rm RR}(A_{\{x_0\}})=0$. Then for every $\varepsilon>0$ and $a\in (A_\af)_{s.a.}$, there exists projections $q_1,\cdots,q_N\in A_{\{x_0\}}$ and complex numbers $\lambda'_1, \cdots,\lambda'_N$ such that 
\[\bigg|\tau(a)-\sum_{i=1}^N\lambda'_i\tau(q_i)\bigg|<\varepsilon\]
for all $\tau\in T(A_\af)$.
\end{lem}
\begin{proof}
Note that  by Theorem \ref{3.7}, $TR(A_\af)=0$. Now fix $\varepsilon>0$ and let $a\in (A_\af)_{s.a.}$. Without loss of generality, assume $\|a\|=1$. Since $A_\af$ has tracial rank zero, it has real rank zero, and hence there exists real numbers $\lambda_i$ and orthogonal projections $p_i\in \mathcal{P}(A_\af)\,(i=1,\cdots,k)$ such that 
\[\|a-\sum_{i=1}^k\lambda_ip_i\|<\varepsilon/8\]
Denote $M=1+\max\{|\lambda_i|: 1\leq i\leq k\}>0$. Let $\mathcal{F}=\{p_1,p_2,\cdots,p_k\}$. Then according to Lemma \ref{7.1}, there exists a projection $p\in A_{\{x_0\}}$ such that 
\[p_ip\in_{\varepsilon/8kM}A_{\{x_0\}}\ \ {\rm and}\ \ \tau(1-p)<\varepsilon^2/4\]
for all $\tau\in T(A_{\{x_0\}})=T(A_\af)$. For every $i\in\{1,\cdots,k\}$, let $b_i\in A_{\{x_0\}}$ be an element with 
\[\|b_i-p_ip\|<\frac{\varepsilon}{8kM}\]
Since ${\rm RR}(A_{\{x_0\}})=0$, there exists, for every $i\in\{1,\cdots,k\}$, a positive integer $m(i)$, complex numbers $\lambda^{(i)}_m$ and projections $p^{(i)}_m\in \mathcal{P}(A_{\{x_0\}})$ for $m=1,\cdots,m(i)$ such that 
\[\|b_i-\sum_{m=1}^{m(i)}\lambda^{(i)}_mp^{(i)}_m\|<\frac{\varepsilon}{4kM}\]
for all $i=1,\cdots,k$. This follows that 
\[\|\sum_{i=1}^k\lambda_ib_i-\sum_{i=1}^k\sum_{m=1}^{m(i)}\lambda_i\lambda^{(i)}_mp^{(i)}_m\|<\sum_{i=1}^k\frac{|\lambda_i|\varepsilon}{4kM}<\frac{\varepsilon}{4}\]
On the other hand, since $\|b_i-p_ip\|<\varepsilon/8kM$, we have
\[\|\sum_{i=1}^k\lambda_ib_i-\sum_{i=1}^k \lambda_ip_ip\|=\|\sum_{i=1}^k(\lambda_ib_i-\lambda_ip_ip)\|<\varepsilon/8\]
This follows that 
\[\|ap-\sum_{i=1}^k\lambda_ib_i\|<\varepsilon/4\]
Denote the projections $p^{(i)}_m\in \mathcal{P}(A_{\{x_0\}})$ and complex numbers $\lambda_i\lambda^{(i)}_m$ for $1\leq m\leq m(i)$ and $1\leq i\leq k$ by $q_i$ and $\lambda'_i$ for $1\leq i\leq N=m(1)+\cdots+m(k)$. Then
\[\|ap-\sum_{i=1}^N\lambda'_iq_i\|<\varepsilon/2\]
This follows that for every $\tau\in T(A_\af)$, we have
\[\bigg|\tau(ap)-\tau\left(\sum_{i=1}^N\lambda'_iq_i\right)\bigg|<\varepsilon/2\]
Note that $\|a\|=1$, and hence $\tau(a^2)=\tau(a^*a)\leq \|a\|^2=1$, which follows that 
\begin{align*}
&|\tau(ap)-\tau(a)|\\
=\,&|\tau(a(1-p))|\\
\leq\,&\sqrt{\tau(a^2)\cdot\tau((1-p)^2)}\\
\leq\,&\sqrt{\tau(1-p)}\\
<\,&\varepsilon/2
\end{align*}
This finally yields
\[\bigg|\tau(a)-\tau\left(\sum_{i=1}^N\lambda'_iq_i\right)\bigg|<\varepsilon\]
for all $\tau\in T(A_\af)$, which completes the proof.
\end{proof}

\begin{lem}\label{7.3}
Assume $x_0\in K$ and ${\rm RR}(A_{\{x_0\}})=0$. Then for every $\varepsilon>0$ and $f\in C(X)_+$, there exists complex numbers $\lambda'_1, \cdots,\lambda'_N$ and $x_1,\cdots,x_N\in K_0(C(X))_+$ such that 
\[|\rho_{A_\af}(j_\af(f))(\tau)-\sum_{i=1}^N\lambda_i'\rho_{A_\af}((j_\af)_{*0}(x_i))(\tau)|<\varepsilon\]
for all $\tau\in T(A_\af)$.
\end{lem}
\begin{proof}
According to Lemma \ref{7.2}, there exists projections $q_1,\cdots,q_N\in A_{\{k_0\}}$ and complex numbers $\lambda'_1, \cdots,\lambda'_N$ such that 
\[\bigg|\tau(j_\af(f))-\sum_{i=1}^N\lambda'_i\tau(q_i)\bigg|<\varepsilon\]
for all $\tau\in T(A_\af)$. Note that the embedding $j_1: C(X)\to A_{\{k_0\}}$ induces a surjection of $K_0$-groups with $(j_1)_{*0}(K_0(C(X))_+)=K_0(A_{\{k_0\}})_+$, there exists $x_i\in K_0(C(X))_+$ such that 
\[(j_1)_{*0}(x_i)=[q_i]\]
Since the embedding $j_{\{k_0\}}: A_{\{k_0\}}\hookrightarrow A_\af$ induces an embedding of $K_0$-groups which sends $x$ to $(x,0)$ by Proposition \ref{3.3}, it follows that
\[(j_\af)_{*0}(x_i)=(j_{\{k_0\}})_{*0}\circ(j_1)_{*0}(x_i)=([q_i],0)\]
and
\[\rho_{A_\af}([q_i])(\tau)=\rho_{A_\af}([q_i],0)(\tau)=\tau(q_i)\]
for all $\tau\in T(A_\af)$ and $1\leq i\leq N$. This implies that
\[\sum_{i=1}^N\lambda'_i\tau(q_i)=\sum_{i=1}^N\lambda'_i\rho_{A_\af}([q_i],0)(\tau)=\sum_{i=1}^N\lambda'_i\rho_{A_\af}((j_\af)_{*0}(x_i))(\tau)\]
for all $\tau\in T(A_\af)$, which completes the proof.
\end{proof}
Now according to Lemma \ref{7.3}, we have, by applying a similar argument as in Lemma \ref{5.10}, the following lemma.
\begin{lem}\label{7.4}
Assume $x_0\in K$ and ${\rm RR}(A_{\{x_0\}})=0$. Let $h: C(X)\hookrightarrow A_\af$ be a unital monomorphism and $h_n: C(X)\hookrightarrow A_\af$ be a sequence of unital monomorphisms. Suppose that for any finite subset $F\subset K_0(C(X))$, there is $N\geq1$ such that
\[\rho_{A_\af}(h_{*0}(x))(\tau)=\rho_{A_\af}({h_n}_{*0}(x))(\tau)\]
for all $n\geq N, \tau\in T(A_\af)$ and $x\in F$, then
\[\rho_{A_\af}(h(f))(\tau)=\lim_{n\to\infty}\rho_{A_\af}(h_n(f))(\tau)\]
for all $f\in C(X)$ and $\tau\in T(A_\af)$. In particular, if $h_1,h_2:C(X)\to A_\af$ are two unital monomorphisms such that
\[\rho_{A_\af}\circ (h_1)_{*0}(x)(\tau)=\rho_{A_\af}\circ(h_2)_{*0}(x)(\tau)\]
for any $x\in K_0(C(X))$ and $\tau\in T(A_\af)$, then
\[\rho_{A_\af}(h_1(f))(\tau)=\rho_{A_\af}(h_2(f))(\tau)\]
for all $f\in C(X)$ and $\tau\in T(A_\af)$.
\end{lem}
\paragraph{Conventions.} We say that $A_{\af}$ (or $A_{\bt}$) has orbit-breaking subalgebras of real rank zero, if there exists $x_0\in K$ such that ${\rm RR}(A_{\{x_0\}})=0$.
\begin{thm}\label{7.5}
Let $\Omega$ be a finite-dimensional compact connected CW-complex with torsion free $K$-groups. Assume that $\Omega$ satisfies the LMP. Let $X=K\times\Omega$ and $\af,\bt: X\to X$ be two minimal rigid homeomorphisms such that $A_\af$ and $A_\bt$ both have orbit-breaking subalgebras of real rank zero. Then the followings are equivalent:

\noindent (1) $(X,\af)$ and $(X,\bt)$ are approximately $K$-conjugate;

\noindent (2) There exists an isomorphism
\[\kappa: (K_0(A_\af), K_0(A_\af)_+, [1_{A_\af}], K_1(A_\af))\to (K_0(A_\bt), K_0(A_\bt)_+, [1_{A_\bt}], K_1(A_\bt))\]
and two sequences of homeomorphisms $\sigma_n,\gamma_n: X\to X$ such that $(X,\af)\sim_{\rm app}(X,\bt)$ via $\sigma_n,\gamma_n$ and for every finitely-generated subgroups $G_i,F_i\subset K_i(C(X))$,
\[\kappa_i\circ (j_\af)_{*i}|_{G_i}=(j_\bt\circ\sigma_n^{-1})_{*i}|_{G_i}\ \ {\rm and}\ \ \kappa_i^{-1}\circ(j_\bt)_{*i}|_{F_i}=(j_\af\circ\gamma_n^{-1})_{*i}|_{F_i}\]
for $i=0,1$ and all sufficient large $n$;

\noindent (3) There exists homeomorphisms $\sigma_n,\gamma_n:X\to X$ such that $(X,\alpha)\sim_{\rm app}(X,\beta)$ via $\sigma_n,\gamma_n$
 and in addition, for any $x\in K_i(C(X))$, there exists $N\geq1$ such that
 \begin{align}
(j_{\alpha}\circ\sigma_n^*)_{*i}(x)=(j_{\alpha}\circ\sigma_{m}^*)_{*i}(x), \ \ \ &(j_{\beta}\circ\gamma_n^*)_{*i}(x)=(j_{\beta}\circ\gamma_{m}^*)_{*i}(x)\\
(j_\alpha\circ(\gamma_n\circ\sigma_m)^*)_{*i}(x)=(j_\alpha)_{*i}(x),\ \ \ &(j_\beta\circ(\sigma_n\circ\gamma_m)^*)_{*i}(x)=(j_\beta)_{*i}(x)
\end{align}
 for all $m\geq n\geq N$ and for $i=0,1$;
 
 \noindent (4) There exists an isomorphism $\theta:A_{\alpha}\to A_\beta$, sequences of unitaries $\{w_n\}\subset A_\alpha$ and $\{u_n\}\subset A_\beta$ and sequences of homeomorphisms $\sigma_n,\gamma_n: X\to X$ such that $(X,\alpha)\sim_{\rm app}(X,\beta)$ via $\sigma_n,\gamma_n$ and in addition, 
 \[\lim_{n\to\infty}\|u_n\theta(j_{\alpha}(f))u_n^*-j_{\beta}(f\circ\gamma_n^{-1})\|=0\]
 and
 \[\lim_{n\to\infty}\|w_n\theta^{-1}(j_{\beta}(f))w_n^*-j_\alpha(f\circ \sigma_n^{-1})\|=0\]
 for all $f\in C(X)$.

\noindent (5) There exists homeomorphisms $\sigma_n,\gamma_n: X\to X$ such that $(X,\alpha)\sim_{\rm app}(X,\beta)$ via $\sigma_n,\gamma_n$ and there exists $\kappa\in KL(A_\alpha, A_\beta)$ which induces an isomorphism
\[\tilde{\kappa}: (K_0(A_\alpha), K_0(A_\alpha)_+, [1_{A_\alpha}], K_1(A_\alpha), T(A_\alpha))\to (K_0(A_\beta), K_0(A_\beta)_+, [1_{A_\beta}], K_1(A_\beta), T(A_\beta))\]
and sequences of unital isomorphisms $\varphi_n, \psi_n$ such that 
\[\lim_{n\to\infty}\|\psi_n(j_\bt(f))-j_\af(f\circ\sigma_n)\|=0, \ \ \ \lim_{n\to\infty}\|\phi_n(j_\af(f))-j_\bt(f\circ\gamma_n)\|=0\]
with $[\phi_n]=[\phi_1]=[\psi_n]^{-1}=\tilde{\kappa}$ for all $n\geq1$.
\end{thm}
\begin{proof}
This is just another version of Theorem \ref{5.13}, where the condition that $K_1(C(\Omega))=0$ is dropped. Note that the assumption that the orbit-breaking subalgebras of $A_\af$ and $A_\bt$ have real rank zero implies that $\af$ and $\bt$ are rigid, by Theorem \ref{3.7}. Then the theorem follows by a similar argument as in Theorem \ref{5.13}, upon applying Lemma \ref{7.4}.
\end{proof}
\begin{rem}
Theorem \ref{7.5} implies that, when the minimal homeomorphisms $\af$ and $\bt$ are rigid and the orbit-breaking subalgebras of $A_\af$ and $A_\bt$ have real rank zero, then the unital asymptotic morphisms in Definition \ref{5.6} can be chosen to be unital isomorphisms. On the other hand, this coincides with another version of definition of approximate $K$-conjugacy made by H. Lin in \cite{LM2}. Also note that when $\Omega=\mathbb{T}$, the rigidity of $\af$ is equivalent to ${\rm RR}(A_{\{x_0\}})=0$ for some $x_0\in K$, as is shown in \cite{LM}. 
\end{rem}

\begin{lem}\label{7.7}
If $(K\times\Omega,\af)$ and $(K\times\Omega,\bt)$ are approximately $K$-conjugate, then 
\[\kappa_i: K_i(A_\af)\to K_i(A_\bt)\]
maps $(j_\af)_{*i}(K_i(C(X)))$ onto $(j_\bt)_{*i}(K_i(C(X)))$ for $i=0,1$.
\end{lem}
\begin{proof}
According to the Pimsner-Voiculescu six term exact sequence, 
\[(j_\af)_{*i}: C(K,{\Z})\otimes K_i(C(\Omega))\cong K_i(C(X))\to K_i(A_\af)\cong [K^0(K,\tilde{\af})\otimes K_i(C(\Omega))]\oplus K_{i+1}(C(\Omega))\]
which maps $f\in C(K,{\Z})\otimes K_i(C(\Omega))=C(K,K_i(C(\Omega))$ to $([f],0)$. According to Theorem \ref{7.5}, for every $f\in C(K,K_i(C(\Omega)))$, there exists $N\geq1$ such that
\[\kappa_i((j_\af)_{*i}(f))=(j_\bt\circ\sigma_n^{-1})_{*i}(f)\in (j_\bt)_{*i}(K_i(C(X)))\]
which implies the result.
\end{proof}
\begin{lem}\label{7.8}
Let $\af=\tilde{\af}\times c_1\in\Raf$ and $\bt=\tilde{\bt}\times c_2\in\Rbt$ be minimal rigid homeomorphisms. Assume that $A_\af$ and $A_\bt$ both have orbit-breaking subalgebras of real rank zero. If $(K\times\Omega,\af)$ and $(K\times\Omega,\bt)$ are $C^*$-strongly approximately  conjugate, then $(K\times\Omega,\af)$ and $(K\times\Omega,\bt)$ are approximately $K$-conjugate.
\end{lem}
\begin{proof}
Let $\phi_n:A_\af\to A_\bt, \psi_n:A_\bt\to A_\af$ be sequences of unital isomorphisms, $\chi_n,\lambda_n: X\to X$ be homeomorphisms such that 
\[[\phi_n]=[\phi_1]=[\psi_n^{-1}]\ \ {\rm in}\ KL(A_\af, A_\bt)\]
and
\[\lim_{n\to\infty}\|\varphi_n\circ j_\alpha(f)-j_\beta\circ(f\circ\chi_n^{-1})\|=0,\ \ \lim_{n\to\infty}\|\psi_n\circ j_\beta(f)-j_\alpha\circ(f\circ\lambda_n^{-1})\|=0\]
Since $K_i(C(\Omega))\,(i=0,1)$ is torsion free,  $K_i(A_\af)$ and $K_i(A_\bt)$ are also torsion free by Proposition \ref{3.3}. Hence $[\phi_n]$ induces a unital isomorphism
\[\tilde{\kappa}: (K_0(A_\af), K_0(A_\af)_+,[1_{A_\af}], K_1(A_\af))\to (K_0(A_\bt), K_0(A_\bt)_+,[1_{A_\bt}], K_1(A_\bt))\]
such that for every $x\in K_i(C(X))$, there exists $N\geq1$ with
\[\tilde{\kappa}_i((j_\af)_{*i}(x))=(j_\bt\circ \lambda_n^*)_{*i}(x)\ \ {\rm and}\ \ \tilde{\kappa}_i^{-1}((j_\bt)_{*i}(x))=(j_\af\circ\chi_n^*)_{*i}(x)\]
Therefore, by a similar argument as the proof of Lemma \ref{6.3} to Lemma \ref{6.8}, and according to Lemma \ref{7.7}, there exists a unital order isomorphism 
\[\Theta: K^0(K,\tilde{\af})\to K^0(K,\tilde{\bt})\]
such that the natural embedding of $K^0(K,\tilde{\af})$ and $K^0(K,\tilde{\bt})$ into $K_i(A_\af)$ and $K_i(A_\bt)$ intertwining $\Theta$ and $\tilde{\kappa}_0$. Note that ${\rm TR}(A_\af)={\rm TR}(A_\bt)=0$, by the classification of simple unital $C^*$-algebras of tracial rank zero, one obtains an isomorphism $\Phi:A_\af\to A_\bt$ such that
\[\Phi_{*0}=\tilde{\kappa}_0,\ \ \ \Phi_{*1}=[\Theta\otimes{\rm id}_{K_1(C(\Omega))}]\oplus {\rm id}_{K_0(C(\Omega))}\ \ {\rm and}\ \ \Phi_{*0}|_{K^0(K,\tilde{\af})\otimes K_0(C(\Omega))}=\Theta\otimes{\rm id}_{K_0(C(\Omega))}\]
Denote $\kappa=\Phi_{*}$. On the other hand, since $\Theta: K^0(K,\tilde{\af})\to K^0(K,\tilde{\bt})$ is a unital order isomorphism, it induces the approximate $K$-conjugacy of $(K,\tilde{\af})$ and $(K,\tilde{\bt})$, say, via $\tilde{\sigma}_n$ and $\tilde{\gamma}_n$. In particular, $(K,\tilde{\af})$ and $(K,\tilde{\bt})$ are weakly approximately conjugate.

Recall that in Lemma \ref{6.10}, $K_1(\Omega)$ is not necessarily assumed to be $0$, which follows that 
\[(K\times\Omega,\tilde{\af}\times c_1)\sim_{\rm app}(K\times\Omega,\tilde{\bt}\times c_2)\ {\rm via\ }\sigma_n,\gamma_n\]
such that $\sigma_n=\tilde{\sigma}_n\times y_n$ and $\gamma_n=\tilde{\gamma}_n\times z_n$ for sequences of cocycles $y_n, z_n\in C(K,G_\mathcal{L})$ for some $\Lip>0$. Then for every 
\[x=\sum_{\rm finite}f_s\otimes x_s\in K_i(C(X))\]
one checks that, there is $N\geq1$ with 
\[\kappa_i\circ(j_\af)_{*i}(x)=(j_\bt\circ\sigma_n^{-1})_{*i}(x)\ \ {\rm and}\ \ \kappa_{i}^{-1}\circ(j_\bt)_{*i}(x)=(j_\af\circ\gamma_n^{-1})_{*i}(x)\]
for all $n\geq N$. Finally, Theorem \ref{7.5} implies that $(K\times\Omega,\af)$ and $(K\times\Omega,\bt)$ are approximately $K$-conjugate (via quasi-trivial cocycles, since $y_n,z_n$ take values in $G_\Lip$).
\end{proof}

\begin{thm}\label{7.9}
Let $(K,\tilde{\af})$ and $(K,\tilde{\bt})$ be Cantor minimal systems. Let $\Omega$ be a compact connected finite CW-complex with torsion free $K$-groups. Assume that $\Omega$ satisfies the LMP. Let $\af=\tilde{\af}\times c_1\in\Raf$ and $\bt=\tilde{\bt}\times c_2\in\Rbt$ be minimal  homeomorphisms of $X$ such that $A_\af$ and $A_\bt$ both have orbit-breaking subalgebras of real rank zero. Denote $X=K\times\Omega$. Then the followings are equivalent:

\noindent $(1)$ $(X,\af)$ and $(X,\bt)$ are approximately $K$-conjugate;

\noindent $(2)$ There exists a homeomorphism $\chi:X\to X$ and a unital order isomorphism
\[\kappa: (K_0(A_\af),K_0(A_\af)_+,[1_{A_\af}], K_1(A_\af))\to (K_0(A_\bt), K_0(A_\bt)_+,[1_{A_\bt}],K_1(A_\bt))\]
such that 
\begin{align}
\begin{cases}
\vspace{0.1cm}
\kappa_i\circ(j_\alpha)_{*i}=(j_\beta\circ\chi^{-1})_{*i},\ \ (i=0,1)\\
\lambda(\tau)\circ j_\af(f)=\tau\circ j_\bt(f\circ\chi^{-1})\ {\rm for\ all\ }\tau\in T(A_\bt)\ {\rm and}\ f\in C(X)
\end{cases}
\end{align}

\noindent $(3)$ $(X,\af)$ and $(X,\bt)$ are $C^*$-strongly approximately conjugate.
\end{thm}
\begin{proof}
$(1)\Rightarrow (2)$: Let $(K\times\Omega,\af)$ and $(K\times\Omega,\bt)$ be approximately $K$-conjugate.  Denote the associated isomorphism of $K$-groups by $\tilde{\kappa}$. According to Lemma \ref{7.7}, $\tilde{\kappa}_i$ maps 
\[(j_\af)_{*i}(K^i(X))\cong K^0(K,\tilde{\af})\otimes K^i(X)\] 
onto 
\[(j_\bt)_{*i}(K^i(X))\cong K^0(K,\tilde{\bt})\otimes K^i(X)\]
According to Lemma \ref{7.8}, there exists a unital order isomorphism $\Theta: K^0(K,\tilde{\af})\to K^0(K,\tilde{\bt})$
such that $\tilde{\kappa}_0|_{(j_\af)_{*0}}=\Theta\otimes {\rm id}_{K_0(C(\Omega))}$
and the natural embedding of $K^0(K,\tilde{\af})$ and $K^0(K,\tilde{\bt})$ into $K_0(A_\af)$ and $K_0(A_\bt)$ intertwining $\Theta$ and $\tilde{\kappa_0}$. Define
\[\kappa_0=\tilde{\kappa}_0\ \ {\rm and}\ \ \kappa_1=(\Theta\otimes{\rm id}_{K_1(C(\Omega))})\oplus {\rm id}_{K_0(C(\Omega))}\]
Then it is clear that $\kappa=(\kappa_0,\kappa_1)$ gives a unital order isomorphism of $(K_0(A_\af), K_1(A_\af))$ and $(K_0(A_\bt), K_1(A_\bt))$.  Let $\tilde{\chi}:K\to K$ be the homeomorphism associated with $\Theta$ and define
\[\chi=\tilde{\chi}\times {\rm id}_\Omega\]
Then it is straightforward to verify that
\[\kappa_i\circ(j_\alpha)_{*i}=(j_\beta\circ\chi^{-1})_{*i}\ \ (i=0,1).\]
We now check that for all $\tau\in T(A_\bt)$ and $f\in C(X)$, we have $\lambda(\tau)\circ j_\af(f)=\tau\circ j_\bt(f\circ\chi^{-1})$. Note that for all $x\in K_0(A_\af)$, $\rho_{A_{\af}}(x)(\lambda(\tau))=\rho_{A_\bt}(\kappa_0(x))(\tau)$, which follows for any $p\in K_0(C(X))$,
\[\rho_{A_\af}((j_\af)_{*0}(p))(\lambda(\tau))=\rho_{A_\bt}(\kappa_0\circ (j_\af)_{*0}(p))(\tau)=\rho_{A_\bt}((j_\bt\circ\chi^{-1})_{*0}(p))(\tau)\]
Since  $TR(A_\af)=TR(A_\bt)=0$, which follows that  there is a unital isomorphism $\phi: A_\af\to A_\bt$ such that $\varphi_{*i}=\kappa_{i}$. Now define two monomorphisms $h_1,h_2: C(X)\hookrightarrow A_\bt$ by $h_1(f)=\varphi\circ j_{\af}(f)$ and $h_2(f)=j_\bt\circ \chi^*(f)$. Then we have
\[\rho_{A_\bt}\circ (h_1)_{*0}(p)(\tau)=\rho_{A_\bt}\circ (h_2)_{*0}(p)(\tau)\]
for all $p\in K_0(C(X))$ and $\tau\in T(A_\bt)$.  Since $A_\af$ and $A_\bt$ have orbit-breaking subalgebras of real rank zero, by Lemma \ref{7.4}, we have
\[\rho_{A_\bt}\circ h_1(f)(\tau)=\rho_{A_\bt}\circ h_2(f)(\tau)\]
for all $f\in C(X)$ and $\tau\in T(A_\bt)$. In other words, $\lambda(\tau)\circ j_\af(f)=\tau\circ j_\bt(f\circ\chi^{-1})$. This shows the implication $(1)\Rightarrow(2)$.
\vspace{0.2cm}

\noindent $(2)\Rightarrow (3)$: Since $A_\af$ and $A_\bt$ have orbit-breaking subalgebras of real rank zero, $\af$ and $\bt$ are rigid by Theorem \ref{3.7}, that is, $TR(A_\alpha)=TR(A_\beta)=0$. According to classification of simple separable amenable unital $C^*$-algebras with tracial rank zero, there is an isomorphism $\Phi: A_\alpha\to A_\beta$ such that $\Phi_{*i}=\kappa_i$. Define two monomorphisms $j_1, j_2: C(X)\hookrightarrow A_\beta$ by 
\[j_1(f)=\Phi\circ j_\alpha(f),\ \ {\rm and}\ \ j_2(f)=j_\beta(f\circ\chi^{-1})\]
Since $\Phi$ induces $\kappa$ and $\lambda:T(A_\beta)\to T(A_\alpha)$ is also induced by $\kappa|_{K_0(A_\alpha)}$, the condition yields that
\begin{align*}
(j_1)_{*i}=(j_2)_{*i} \ \ {\rm and}\ \ \tau\circ h_1(f)=\tau\circ h_2(f)
\end{align*}
for all $\tau\in T(A_\beta)$ and $f\in C(X)$. According to Lemma \ref{5.12}, $h_1$ and $h_2$ are approximately unitarily equivalent, that is, there is a sequence of unitaries $\{u_n\}\subset A_\beta$ such that
\[\lim_{n\to\infty}\|u_n^*\Phi(j_\alpha(f))u_n-j_\beta(f\circ\chi^{-1})\|=0\]
Finally, Let $\varphi_n={\rm Ad}_{u_n}\circ\Phi, \psi_n={\rm Ad}_{u_n^*}\circ\Phi^{-1}$ and $\chi_n=\chi^*, \lambda_n=(\chi^*)^{-1}$, then it is immediate that $\varphi_n,\psi_n,\chi_n,\lambda_n$ satisfy the conditions in Definition \ref{5.11}, that is, $(X,\af)$ and $(X,\bt)$ are $C^*$-strongly approximately conjugate.
\vspace{0.2cm}

\noindent $(3)\Rightarrow (1)$: This follows immediately from Lemma \ref{7.8}.
\end{proof}

\vspace{1cm}

\noindent{\it Correspondence to be sent to: sihanwei2093@yeah.net}

\end{document}